\documentclass[a4paper,12pt,intlimits,oneside]{amsart}

\usepackage{amsmath}
\usepackage{amssymb}
\usepackage{amsfonts}
\usepackage{amsthm,amscd}

\newtheorem{thm}{Theorem}[section]
\newtheorem{prop}[thm]{Proposition}
\newtheorem{cor}[thm]{Corollary}

\newtheorem{defn}[thm]{Definition}
\newtheorem{remark}[thm]{Remark}

\theoremstyle{definition}
\newcommand{\comment}[1]{}

\numberwithin{equation}{section}

\def\lsim{\raisebox{-1ex}{$~\stackrel{\textstyle <}{\sim}~$}}

\theoremstyle{definition}

\begin{document}
\title[Duals of Hardy-amalgam spaces]{New characterization of the duals of Hardy-amalgam spaces}
\author[Z.V.P. Abl\'e]{Zobo Vincent de Paul Abl\'e}
\address{Laboratoire de Math\'ematiques Fondamentales, UFR Math\'ematiques et Informatique, Universit\'e F\'elix Houphou\"et-Boigny Abidjan-Cocody, 22 B.P 582 Abidjan 22. C\^ote d'Ivoire}
\email{{\tt vincentdepaulzobo@yahoo.fr}}
\author[J. Feuto]{Justin Feuto}
\address{Laboratoire de Math\'ematiques Fondamentales, UFR Math\'ematiques et Informatique, Universit\'e F\'elix Houphou\"et-Boigny Abidjan-Cocody, 22 B.P 1194 Abidjan 22. C\^ote d'Ivoire}
\email{{\tt justfeuto@yahoo.fr}}



\date{}
\begin{abstract}
In this paper, carrying on with our study of the Hardy-amalgam spaces $\mathcal H^{(q,p)}$ and $\mathcal{H}_{\mathrm{loc}}^{(q,p)}$ ($0<q,p<+\infty$), we give a characterization of their duals whenever $0<q\leq 1<p<+\infty$. Moreover, when $0<q\leq p\leq 1$, these characterizations coincide with those obtained in our earlier papers. Thus, we obtain a unified characterization of the dual spaces of Hardy-amalgam spaces, when $0<q\leq1$ and $0<q\leq p<+\infty$. 
\end{abstract}

\maketitle

\section{Introduction}
Let $\varphi\in\mathcal C^\infty(\mathbb R^d)$ with support on $B(0,1)$ such that $\int_{\mathbb R^d}\varphi dx=1$, where $B(0,1)$ is the unit open ball centered at $0$ and $\mathcal C^\infty(\mathbb R^d)$ denotes the space of infinitely differentiable complex valued functions on $\mathbb R^d$. For $t>0$, we denote by $\varphi_t$ the dilated function $\varphi_t(x)=t^{-d}\varphi(x/t)$, $x\in\mathbb R^d$. The Hardy-amalgam spaces $\mathcal H^{(q,p)}$ and $\mathcal{H}_{\mathrm{loc}}^{(q,p)}$ ($0<q,p<+\infty$) introduced in \cite{AbFt}, are a generalization of the classical Hardy spaces $\mathcal H^q$ and $\mathcal{H}_{\mathrm{loc}}^q$ in the sense that they are respectively the spaces of tempered distributions $f$ such that the maximal functions
\begin{equation}
\mathcal M_{\varphi}(f):=\sup_{t>0}|f\ast\varphi_t|\ \ \ \text{ and }\ \ \ {\mathcal{M}_{\mathrm{loc}}}_{_{\varphi}}(f):=\sup_{0<t\leq 1}|f\ast\varphi_t|\ \ \ \ \ \label{maximal}
\end{equation}
belong to the Wiener amalgam spaces $(L^q,\ell^p)(\mathbb R^d)=:(L^q,\ell^p)$. 

We recall that for $0<p,q\leq+\infty$,  a locally integrable function $f$ belongs to the amalgam space $(L^q,\ell^p)$ if 
$$\left\|f\right\|_{q,p}:=\left\|\left\{\left\|f\chi_{_{Q_k}}\right\|_{q}\right\}_{k\in\mathbb{Z}^d}\right\|_{\ell^{p}}<+\infty,$$ 
where $Q_k=k+\left[0,1\right)^{d}$ for $k\in\mathbb Z^d$.
Endowed with the (quasi)-norm $\left\|\cdot\right\|_{q,p}$, the amalgam space $(L^q,\ell^p)$ is a complete space, and a Banach space for $1\leq q, p\leq+\infty$. These spaces coincide with the Lebesgue spaces whenever $p=q$, but also the class of these spaces decreases with respect to the local component and increases with respect to the global component. Hardy-amalgam spaces are equal to $(L^q,\ell^p)$ spaces for $1<p,q<+\infty$ and the dual space of $(L^q,\ell^p)$ is known to be $(L^{q'},\ell^{p'})$ under the assumption that $1\leq q,p<+\infty$, where $\frac{1}{q}+\frac{1}{q'}=1$ and $\frac{1}{p}+\frac{1}{p'}=1$, with the usual conventions (see \cite{BDD}, \cite{RBHS}, \cite{FSTW}, \cite{FH} and \cite{JSTW}).

We also know that the Hardy-Littlewood maximal operator $\mathfrak{M}$ defined for a locally integrable function $f$ by 
\begin{eqnarray}
\mathfrak{M}(f)(x)=\underset{r>0}\sup\ |B(x,r)|^{-1}\int_{B(x,r)}|f(y)|dy,\ \ x\in\mathbb{R}^d.\label{maxop}
\end{eqnarray} 
is bounded in $(L^q,\ell^p)$ ($1<q,p<+\infty$) (see \cite{CLHS} Theorems 4.2 and 4.5).

As for the classical Hardy spaces, not only the definition does not depend on the particular function $\varphi$, but the Hardy-amalgam spaces can be characterized in terms of grand maximal function, and we can also replace the regular function $\varphi$ by the Poisson kernel. In the case $0<q\leq\min\left\{1,p\right\}$ and $0<p<+\infty$, we proved in \cite{AbFt} and \cite{AbFt2} that $\mathcal H^{(q,p)}$ and $\mathcal{H}_{\mathrm{loc}}^{(q,p)}$ have an atomic characterization, with atoms which are exactly those of classical Hardy spaces. Moreover, when $0<q\leq p\leq 1$, we obtained a characterization of the duals of these spaces in \cite{AbFt1} and \cite{AbFt2}.\\

The aim of this paper is to characterize the dual of $\mathcal H^{(q,p)}$ and $\mathcal{H}_{\mathrm{loc}}^{(q,p)}$ spaces when  $0<q\leq 1<p<+\infty$. For this purpose, we organize this paper as follows. 

In Section 2, we recall some properties of the Wiener amalgam spaces $(L^q,\ell^p)$ and, the Hardy-amalgam spaces $\mathcal{H}^{(q,p)}$ and $\mathcal{H}_{\mathrm{loc}}^{(q,p)}$ obtained in \cite{AbFt}, \cite{AbFt1} and \cite{AbFt2} that we will need. 

Next, in Section 3, inspired by the approach developed by Wael Abu-Shammala in his thesis \cite{AT} to characterize the dual space of the Hardy-Lorentz spaces $H^{q,p}$ for $0<q<1<p<+\infty$, we characterize the dual space of the Hardy-amalgam $\mathcal{H}^{(q,p)}$ for $0<q\leq 1<p<+\infty$. Moreover, whenever $0<q\leq p\leq 1$, we prove that it coincides with the one obtained in \cite{AbFt1}, unifying so the characterization of the dual of $\mathcal{H}^{(q,p)}$ spaces for $0<q\leq1$ and $0<q\leq p<+\infty$. 

Finally, in the last section, we give a characterization of the dual of $\mathcal{H}_{\mathrm{loc}}^{(q,p)}$ for $0<q\leq 1<p<+\infty$. As in the case of $\mathcal{H}^{(q,p)}$, we obtain a unified characterization of the dual of $\mathcal{H}_{\mathrm{loc}}^{(q,p)}$ spaces for $0<q\leq1$ and $0<q\leq p<+\infty$. As an application of this result, whenever $0<q\leq1$ and $0<q\leq p<+\infty$, we extend a result of boundedness of pseudo-differential operators obtained in \cite{AbFt2}.\\ 

Throughout the paper, we always let $\mathbb{N}=\left\{1,2,\ldots\right\}$ and $\mathbb{Z}_{+}=\left\{0,1,2,\ldots\right\}$. We use $\mathcal S := \mathcal S(\mathbb R^{d})$ to denote the Schwartz class of rapidly decreasing smooth functions equipped with the topology defined by the family of norms $\left\{\mathcal{N}_{m}\right\}_{m\in\mathbb{Z}_{+}}$, where for all $m\in\mathbb{Z}_{+}$ and $\psi\in\mathcal{S}$, $$\mathcal{N}_{m}(\psi):=\underset{x\in\mathbb R^{d}}\sup(1 + |x|)^{m}\underset{|\beta|\leq m}\sum|{\partial}^\beta \psi(x)|$$ with $|\beta|=\beta_1+\ldots+\beta_d$, ${\partial}^\beta=\left(\partial/{\partial x_1}\right)^{\beta_1}\ldots\left(\partial/{\partial x_d}\right)^{\beta_d}$ for all $\beta=(\beta_1,\ldots,\beta_d)\in\mathbb{Z}_{+}^d$ and $|x|:=(x_1^2+\ldots+x_d^2)^{1/2}$. The dual space of $\mathcal S$ is the space of tempered distributions denoted by $\mathcal S':= \mathcal S'(\mathbb R^{d})$ equipped with the weak-$\ast$ topology. If $f\in\mathcal{S'}$ and $\theta\in\mathcal{S}$, we denote the evaluation of $f$ on $\theta$ by $\left\langle f,\theta\right\rangle$. The letter $C$ will be used for non-negative constants independent of the relevant variables that may change from one occurrence to another. When a constant depends on some important parameters $\alpha,\gamma,\ldots$, we denote it by $C(\alpha,\gamma,\ldots)$. Constants with subscript, such as $C_{\alpha,\gamma,\ldots}$, do not change in different occurrences but depend on the parameters mentioned in it. We propose the following abbreviation $\mathrm{\bf A}\lsim \mathrm{\bf B}$ for the inequalities $\mathrm{\bf A}\leq C\mathrm{\bf B}$, where $C$ is a positive constant independent of the main parameters. If $\mathrm{\bf A}\lsim \mathrm{\bf B}$ and $\mathrm{\bf B}\lsim \mathrm{\bf A}$, then we write $\mathrm{\bf A}\approx \mathrm{\bf B}$. For any given quasi-normed spaces $\mathcal{A}$ and $\mathcal{B}$ with the corresponding quasi-norms $\left\|\cdot\right\|_{\mathcal{A}}$ and $\left\|\cdot\right\|_{\mathcal{B}}$, the symbol $\mathcal{A}\hookrightarrow\mathcal{B}$ means that if $f\in\mathcal{A}$, then $f\in\mathcal{B}$ and $\left\|f\right\|_{\mathcal{B}}\lsim\left\|f\right\|_{\mathcal{A}}$. Also, $\mathcal{A}\cong\mathcal{B}$ means that $\mathcal{A}$ is isomorphic to $\mathcal{B}$, with equivalence of the quasi-norms $\left\|\cdot\right\|_{\mathcal{A}}$ and $\left\|\cdot\right\|_{\mathcal{B}}$.

For $\lambda>0$ and a cube $Q\subset\mathbb R^{d}$ (by a cube we mean a cube whose edges are parallel to the coordinate axes), we write $\lambda Q$ for the cube with same center as $Q$ and side-length $\lambda$ times side-length of $Q$, while $\left\lfloor \lambda \right\rfloor$ stands for the greatest integer less or equal to $\lambda$. Also, for $x\in\mathbb R^{d}$ and $\ell>0$, $Q(x,\ell)$ will denote the cube centered at $x$ and side-length $\ell$. We use the same notations for balls. For a measurable set $E\subset\mathbb R^d$, we denote by $\chi_{_{E}}$ the characteristic function of $E$ and $\left|E\right|$ for its Lebesgue measure. To finish, we denote by $\mathcal{Q}$ the set of all cubes.

\section{Prerequisites on Wiener amalgam and Hardy-amalgam spaces } 

\subsection{Wiener amalgam spaces $(L^q,\ell^p)$} 

Let  $0<q<1$ and $0<p\leq 1$. We have the following reverse Minkowski's inequality for $(L^q, \ell^p)$ proved in \cite{AbFt1}.

\begin{prop} \label{InverseHoldMink3}
Let $0<q<1$ and $0<p\leq 1$. For all finite sequence $\left\{f_n\right\}_{n=0}^m$ of elements of $(L^q, \ell^p)$, we have
\begin{eqnarray}
\sum_{n=0}^m\left\|f_n\right\|_{q,p}\leq\left\|\sum_{n=0}^m|f_n|\right\|_{q,p}. \label{InverseHoldMink4}
\end{eqnarray}
\end{prop}

Notice that this proposition holds for any sequence in $(L^q, \ell^p)$, for $0<q<1$ and $0<p\leq 1$, but also for $q=1$.

Another very useful result is the following. 

\begin{prop}[\cite{LSUYY}, Proposition 11.12] \label{operamaxima}
Let $1<u\leq+\infty$ and $1<q,p<+\infty$. Then, for all sequence of measurable functions $\left\{f_n\right\}_{n\geq 0}$,
\begin{eqnarray*}
\left\|\left(\sum_{n\geq0}\left[\mathfrak{M}(f_n)\right]^{u}\right)^{\frac{1}{u}}\right\|_{q,p}\approx\left\|\left(\sum_{n\geq0}|f_n|^{u}\right)^{\frac{1}{u}}\right\|_{q,p}, 
\end{eqnarray*}
with the implicit equivalent positive constants independent of $\left\{f_n\right\}_{n\geq 0}$ .
\end{prop}

\subsection{Hardy-amalgam spaces $\mathcal{H}^{(q,p)}$ and $\mathcal{H}_{\mathrm{loc}}^{(q,p)}$} 

Let $0<q,p<+\infty$. The Hardy-amalgam spaces $\mathcal{H}^{(q,p)}$ and $\mathcal{H}_{\mathrm{loc}}^{(q,p)}$ are Banach spaces whenever $1\leq q,p<+\infty$ and quasi-Banach spaces otherwise (see \cite{AbFt}, Proposition 3.8). Moreover, for $0<q\leq 1$ and $q\leq p<+\infty$, we obtained an atomic decomposition for these spaces, more precisely, reconstruction and decomposition theorems (see Theorems 4.3, 4.4 and 4.6 in \cite{AbFt}, and Theorems 3.2, 3.3 and 3.9 in \cite{AbFt2}) by using the following definition of an atom which is also the one of an atom for $\mathcal{H}^{q}$.

\begin{defn}\label{defhqpatom}
Let $1<r\leq+\infty$ and $s\geq\left\lfloor d\left(\frac{1}{q}-1\right)\right\rfloor$ be an integer. A function $\textbf{a}$ is a $(q,r,s)$-atom on $\mathbb{R}^d$ for $\mathcal{H}^{(q,p)}$ if there exists a cube $Q$ such that  
\begin{enumerate}
\item $\text{supp}(\textbf{a})\subset Q$;
\item $\left\|\textbf{a}\right\|_r\leq|Q|^{\frac{1}{r}-\frac{1}{q}}$; \label{defratom1}
\item $\int_{\mathbb{R}^d}x^{\beta}\textbf{a}(x)dx=0$, for all multi-indexes $\beta$ with $|\beta|\leq s$. \label{defratom2}
\end{enumerate}
\end{defn}

For $\mathcal{H}_{\mathrm{loc}}^{(q,p)}$ spaces, Condition (\ref{defratom2}) is not requiered when $|Q|\geq1$. We denote by $\mathcal{A}(q,r,s)$ the set of all $(\textbf{a},Q)$ such that $\textbf{a}$ is a $(q,r,s)$-atom and $Q$ is the associated cube (with respect to Definition \ref{defhqpatom}), and $\mathcal{A}_{\mathrm{loc}}(q,r,s)$ for the local version $\mathcal{H}_{\mathrm{loc}}^{(q,p)}$. 

Let $1<r\leq+\infty$ and $\delta\geq\left\lfloor d\left(\frac{1}{q}-1\right)\right\rfloor$ be an integer. For simplicity, we denote by $\mathcal{H}_{fin}^{(q,p)}$ and $\mathcal{H}_{\mathrm{loc},fin}^{(q,p)}$ respectively the subspaces of $\mathcal{H}^{(q,p)}$ and $\mathcal{H}_{\mathrm{loc}}^{(q,p)}$ consisting of finite linear combinations of $(q,r,\delta)$-atoms. Then, $\mathcal{H}_{fin}^{(q,p)}$ and $\mathcal{H}_{\mathrm{loc},fin}^{(q,p)}$ are respectively dense subspaces of $\mathcal{H}^{(q,p)}$ and $\mathcal{H}_{\mathrm{loc}}^{(q,p)}$ under the quasi-norms $\left\|\cdot\right\|_{\mathcal{H}^{(q,p)}}$ and $\left\|\cdot\right\|_{\mathcal{H}_{\mathrm{loc}}^{(q,p)}}$ (see Remark 4.7 in \cite{AbFt} and Remark 3.12 in \cite{AbFt2}).

\subsection{Campanato spaces and duality for $\mathcal{H}^{(q,p)}$ and $\mathcal{H}_{\mathrm{loc}}^{(q,p)}$} 

For $0<q\leq1$ and $0<p<+\infty$, we consider an integer $\delta\geq\left\lfloor d\left(\frac{1}{q}-1\right)\right\rfloor$.  
Let $1<r\leq+\infty$. We denote by $L_{\mathrm{comp}}^r(\mathbb{R}^d)$ the subspace of $L^r$- functions with compact support. We set\\ 
$L_{\mathrm{comp}}^{r,\delta}(\mathbb{R}^d):=\left\{f\in L_{\mathrm{comp}}^r(\mathbb{R}^d):\ \int_{\mathbb{R}^d}f(x)x^{\beta}dx=0,\ |\beta|\leq\delta\right\}$ and, for all cube $Q$,\\ $L^{r,\delta}(Q):=\left\{f\in L^r(Q):\ \int_{Q}f(x)x^{\beta}dx=0,\ |\beta|\leq\delta\right\}$, where $L^r(Q)$ stands for the subspace of $L^r$-functions supported in $Q$. Also, $L^r(Q)$ is endowed with the norm $\left\|\cdot\right\|_{L^r(Q)}$ defined by $\left\|f\right\|_{L^r(Q)}:=\left(\int_{Q}|f(x)|^rdx\right)^{\frac{1}{r}}$, with the usual modification when $r=+\infty$.
 We point out that $L_{\mathrm{comp}}^{r,\delta}(\mathbb{R}^d)\\=\mathcal{H}_{fin}^{(q,p)}$, when $0<q\leq1$ and $0<q\leq p<+\infty$, where $\mathcal{H}_{fin}^{(q,p)}$ is the subspace of $\mathcal{H}^{(q,p)}$ consisting of finite linear combinations of $(q,r,\delta)$-atoms. 

Let $f\in L_{\mathrm{loc}}^1$ and $Q$ be a cube. Then, there exists a unique  polynomial of $\mathcal{P_{\delta}}$ ($\mathcal{P_{\delta}}:=\mathcal{P_{\delta}}(\mathbb{R}^d)$ is the space of polynomial functions of degree at most $\delta$) that we denote by $P_Q^{\delta}(f)$ such that, for all $\mathfrak{q}\in\mathcal{P_{\delta}}$,
\begin{eqnarray}
\int_{Q}\left[f(x)-P_Q^{\delta}(f)(x)\right]\mathfrak{q}(x)dx=0. \label{Campanato1}
\end{eqnarray} 
Moreover, we have $P_Q^{\delta}(f)=f$, if $f\in\mathcal{P_{\delta}}$.   

\begin{remark} \label{0laremarquetile0}
Let $1\leq r\leq+\infty$ and $f\in L_{\mathrm{loc}}^r$. Then, for all $Q\in\mathcal{Q}$,    
\begin{eqnarray}
\sup_{x\in Q}|P_Q^{\delta}(f)(x)|\leq\frac{C}{|Q|}\int_{Q}|f(x)|dx, \label{1Campanato3}
\end{eqnarray}
and
\begin{eqnarray}
\left(\int_{Q}|P_Q^{\delta}(f)(x)|^rdx\right)^{\frac{1}{r}}\leq C\left(\int_{Q}|f(x)|^rdx\right)^{\frac{1}{r}}, \label{3Campanato3}
\end{eqnarray}
where the constant $C>0$ does not depend on $Q$ and $f$, with the usual modification when $r=+\infty$. 
\end{remark}

For this remark, see \cite{SZLU}, Lemma 4.1.  

\begin{defn} [\cite{NEYS}, Definition 6.1] Let $1\leq r\leq+\infty$, a function $\phi: \mathcal{Q}\rightarrow (0,+\infty)$  and $f\in L_{\mathrm{loc}}^r$. One denotes 
\begin{equation}
\left\|f\right\|_{\mathcal{L}_{r,\phi,\delta}}:=\sup_{Q\in\mathcal{Q}}\frac{1}{\phi(Q)}\left(\frac{1}{|Q|}\int_{Q}\left|f(x)-P_Q^{\delta}(f)(x)\right|^rdx\right)^{\frac{1}{r}}, \label{Campanato2}
\end{equation}
when $r<+\infty$, and
\begin{eqnarray}
\left\|f\right\|_{\mathcal{L}_{r,\phi,\delta}}:=\sup_{Q\in\mathcal{Q}}\frac{1}{\phi(Q)}\left\|f-P_Q^{\delta}(f)\right\|_{L^{\infty}(Q)}, \label{Campanato3}
\end{eqnarray}
when $r=+\infty$. Then, the Campanato space $\mathcal{L}_{r,\phi,\delta}(\mathbb{R}^d)$ is defined to be the set of all $f\in L_{\mathrm{loc}}^r$ such that $\left\|f\right\|_{\mathcal{L}_{r,\phi,\delta}}<+\infty$. One considers elements in $\mathcal{L}_{r,\phi,\delta}(\mathbb{R}^d)$ modulo polynomials of degree $\delta$ so that $\mathcal{L}_{r,\phi,\delta}(\mathbb{R}^d)$ is a Banach space. When one writes $f\in\mathcal{L}_{r,\phi,\delta}(\mathbb{R}^d)$, then $f$ stands for the representative of $\left\{f+\mathfrak{q}:\ \mathfrak{q}\ \text{ is polynomial of degree}\ \delta\right\}$.  
\end{defn}

For simplicity, we just denote $\mathcal{L}_{r,\phi,\delta}(\mathbb{R}^d)$ by $\mathcal{L}_{r,\phi,\delta}$ and for any $T\in\left(\mathcal{H}^{(q,p)}\right)^{\ast}$ the topological dual of $\mathcal{H}^{(q,p)}$, we put 
$$\left\|T\right\|:=\left\|T\right\|_{\left(\mathcal{H}^{(q,p)}\right)^{\ast}}=\sup_{\underset{\left\|f\right\|_{\mathcal{H}^{(q,p)}}\leq 1}{f\in\mathcal{H}^{(q,p)}}}|T(f)|.$$ 

Also, following \cite{YDYS}, Definition 7.1, p. 51, we introduced the following definition in \cite{AbFt2}.

\begin{defn}\label{defindudualloc}  
Let $1\leq r\leq+\infty$ and $\phi: \mathcal{Q}\rightarrow (0,+\infty)$ be a function. The space $\mathcal{L}_{r,\phi,\delta}^{\mathrm{loc}}:=\mathcal{L}_{r,\phi,\delta}^{\mathrm{loc}}(\mathbb{R}^d)$ is the set of all $f\in L_{\mathrm{loc}}^r$ such that $\left\|f\right\|_{{\mathcal{L}}_{r,\phi,\delta}^{\mathrm{loc}}}<+\infty$, where 
\begin{eqnarray*}
\left\|f\right\|_{{\mathcal{L}}_{r,\phi,\delta}^{\mathrm{loc}}}:&=&\sup_{\underset{|Q|\geq1}{Q\in\mathcal{Q}}}\frac{1}{\phi(Q)}\left(\frac{1}{|Q|}\int_{Q}|f(x)|^rdx\right)^{\frac{1}{r}}\\
&+&\sup_{\underset{|Q|<1}{Q\in\mathcal{Q}}}\frac{1}{\phi(Q)}\left(\frac{1}{|Q|}\int_{Q}\left|f(x)-P_Q^{\delta}(f)(x)\right|^rdx\right)^{\frac{1}{r}},
\end{eqnarray*}
when $r<+\infty$, and
\begin{eqnarray*}
\left\|f\right\|_{\mathcal{L}_{r,\phi,\delta}^{\mathrm{loc}}}:=\sup_{\underset{|Q|\geq1}{Q\in\mathcal{Q}}}\frac{1}{\phi(Q)}\left\|f\right\|_{L^{\infty}(Q)}+\sup_{\underset{|Q|<1}{Q\in\mathcal{Q}}}\frac{1}{\phi(Q)}\left\|f-P_Q^{\delta}(f)\right\|_{L^{\infty}(Q)}, 
\end{eqnarray*}
when $r=+\infty$.
\end{defn} 

We note that $\left\|f\right\|_{{\mathcal{L}}_{r,\phi,\delta}}\leq (1+C)\left\|f\right\|_{{\mathcal{L}}_{r,\phi,\delta}^{\mathrm{loc}}}$, for any $f\in L_{\mathrm{loc}}^r$, where the constant $C>0$ is the one of Remark \ref{0laremarquetile0}. Therefore, ${\mathcal{L}}_{r,\phi,\delta}^{\mathrm{loc}}\hookrightarrow {\mathcal{L}}_{r,\phi,\delta}$, since $\left\|\cdot\right\|_{{\mathcal{L}}_{r,\phi,\delta}^{\mathrm{loc}}}$ defines a norm on ${\mathcal{L}}_{r,\phi,\delta}^{\mathrm{loc}}$. 

For all $T\in\left(\mathcal{H}_{\mathrm{loc}}^{(q,p)}\right)^{\ast}$ the topological dual of $\mathcal{H}_{\mathrm{loc}}^{(q,p)}$, we set $$\left\|T\right\|:=\left\|T\right\|_{\left(\mathcal{H}_{\mathrm{loc}}^{(q,p)}\right)^{\ast}}=\sup_{\underset{\left\|f\right\|_{\mathcal{H}_{\mathrm{loc}}^{(q,p)}}\leq 1}{f\in\mathcal{H}_{\mathrm{loc}}^{(q,p)}}}|T(f)|.$$  

We consider the functions $\phi_{1},\ \phi_{2}\ \text{and}\ \phi_{3}:\mathcal{Q}\rightarrow(0,+\infty)$ defined by 
\begin{equation}
\phi_{1}(Q)=\frac{\left\|\chi_{_{Q}}\right\|_{q,p}}{|Q|}\ ,\ \phi_{2}(Q)=\frac{\left\|\chi_{_{Q}}\right\|_q}{|Q|}\ \text{ and }\ \phi_{3}(Q)=\frac{\left\|\chi_{_{Q}}\right\|_p}{|Q|}\ , 
\label{Campanatonolocnloc}
\end{equation} 
for all $Q\in\mathcal{Q}$, whenever $0<q\leq 1$ and $0<p<+\infty$. In \cite{AbFt1} and \cite{AbFt2}, we obtained the following duality results for $\mathcal{H}^{(q,p)}$ and $\mathcal{H}_{\mathrm{loc}}^{(q,p)}$, when $0<q\leq p\leq 1$. 

\begin{thm} [\cite{AbFt1}, Theorem 3.3] \label{theoremdual}
Suppose that $0<q\leq p\leq 1$. Let $1<r\leq+\infty$. Then, $\left(\mathcal{H}^{(q,p)}\right)^{\ast}$ is isomorphic to $\mathcal{L}_{r',\phi_1,\delta}$ with equivalent norms, where $\frac{1}{r}+\frac{1}{r'}=1$. More precisely, we have the following assertions:  
\begin{enumerate}
\item Let $g\in\mathcal{L}_{r',\phi_1,\delta}$. Then, the mapping 
$$T_g:f\in\mathcal{H}_{fin}^{(q,p)}\longmapsto\int_{\mathbb{R}^d}g(x)f(x)dx,$$ where $\mathcal{H}_{fin}^{(q,p)}$ is the subspace of $\mathcal{H}^{(q,p)}$ consisting of finite linear combinations of $(q,r,\delta)$-atoms, extends to a unique continuous linear functional on $\mathcal{H}^{(q,p)}$ such that
\begin{eqnarray*}
\left\|T_g\right\|\leq C\left\|g\right\|_{\mathcal{L}_{r',\phi_{1},\delta}}, 
\end{eqnarray*}
where $C>0$ is a constant independent of $g$. \label{dualpoint1}
\item Conversely, for any $T\in\left(\mathcal{H}^{(q,p)}\right)^{\ast}$, there exists $g\in\mathcal{L}_{r',\phi_1,\delta}$ such that $$T(f)=\int_{\mathbb{R}^d}g(x)f(x)dx,\ \text{ for all }\ f\in\mathcal{H}_{fin}^{(q,p)},$$ and 
\begin{eqnarray*}
\left\|g\right\|_{\mathcal{L}_{r',\phi_{1},\delta}}\leq C\left\|T\right\|, 
\end{eqnarray*}
where $C>0$ is a constant independent of $T$. \label{dualpoint2}
\end{enumerate}
\end{thm} 

\begin{thm}[\cite{AbFt2}, Theorem 4.3] \label{theoremdualloc}
Suppose that $0<q\leq p\leq 1$. Let $1<r\leq+\infty$. Then, $\left(\mathcal{H}_{\mathrm{loc}}^{(q,p)}\right)^{\ast}$ is isomorphic to $\mathcal{L}_{r',\phi_1,\delta}^{\mathrm{loc}}$ with equivalent norms, where $\frac{1}{r}+\frac{1}{r'}=1$. More precisely, we have the following assertions: 
\begin{enumerate}
\item Let $g\in\mathcal{L}_{r',\phi_1,\delta}^{\mathrm{loc}}$. Then, the mapping
$$T_g:f\in\mathcal{H}_{\mathrm{loc},fin}^{(q,p)}\longmapsto\int_{\mathbb{R}^d}g(x)f(x)dx,$$ where $\mathcal{H}_{\mathrm{loc},fin}^{(q,p)}$ is the subspace of $\mathcal{H}_{\mathrm{loc}}^{(q,p)}$ consisting of finite linear combinations of $(q,r,\delta)$-atoms, extends to a unique continuous linear functional on $\mathcal{H}_{\mathrm{loc}}^{(q,p)}$ such that
\begin{eqnarray*}
\left\|T_g\right\|\leq C\left\|g\right\|_{\mathcal{L}_{r',\phi_{1},\delta}^{\mathrm{loc}}}, 
\end{eqnarray*}
 where $C>0$ is a constant independent of $g$. \label{1theoremdualloc}
\item Conversely, for any $T\in\left(\mathcal{H}_{\mathrm{loc}}^{(q,p)}\right)^{\ast}$, there exists $g\in\mathcal{L}_{r',\phi_1,\delta}^{\mathrm{loc}}$ such that $$T(f)=\int_{\mathbb{R}^d}g(x)f(x)dx,\ \text{ for all }\ f\in\mathcal{H}_{\mathrm{loc},fin}^{(q,p)}\ ,$$
and 
\begin{eqnarray*}
\left\|g\right\|_{\mathcal{L}_{r',\phi_{1},\delta}^{\mathrm{loc}}}\leq C\left\|T\right\|, 
\end{eqnarray*}
 where $C>0$ is a constant independent of $T$. \label{2theoremdualloc}
\end{enumerate}
\end{thm}

\section{Duality $\mathcal{H}^{(q,p)}-\mathcal{L}_{r,\phi_{1},\delta}^{(q,p,\eta)}$}

In this section, we give a characterization of the topological dual of $\mathcal{H}^{(q,p)}$, when $0<q\leq1<p<+\infty$. The idea of this characterization follows from the approach developed by Wael Abu-Shammala in his thesis \cite{AT} (see p. 32). Furthermore, we obtain a unified characterization of the dual of $\mathcal{H}^{(q,p)}$ for $0<q\leq1$ and $q\leq p<+\infty$.\\ 

Let $1\leq r<+\infty$, $g$ be a function in $L_{\mathrm{loc}}^r$ and $\Omega$ be an open subset such that $\Omega\neq\mathbb{R}^d$. We define $$\textit{O}(g,\Omega,r):=\sup{\sum_{n\geq0}|\widetilde{Q^n}|^{\frac{1}{r'}}\left(\int_{\widetilde{Q^n}}\left|g(x)-P_{\widetilde{Q^n}}^{\delta}(g)(x)\right|^{r}dx\right)^{\frac{1}{r}}},$$ where $\frac{1}{r}+\frac{1}{r'}=1$ and the supremum is taken over all families of cubes $\left\{Q^n\right\}_{n\geq0}$ such that $Q^n\subset\Omega$, for all $n\geq0$ and $\sum_{n\geq0}\chi_{_{Q^n}}\leq K(d)$, with $\widetilde{Q^n}=C_0Q^n$, where $K(d)>0$ and $C_0>1$ are some constants independent of $\Omega$ and $\left\{Q^n\right\}_{n\geq0}$.  

Given an open cube $Q$, we consider the family of cubes reduced to the cube $C_0^{-1}Q$; namely $\left\{C_0^{-1}Q\right\}$. Then, $C_0^{-1}Q\subset Q$ and $\chi_{C_0^{-1}Q}<K(d)$, with $\widetilde{Q}:=C_0(C_0^{-1}Q)=Q$, and so, by definition of $\textit{O}(g,Q,r)$, we have  
\begin{eqnarray*}
\textit{O}(g,Q,r)&\geq&|Q|^{\frac{1}{r'}}\left(\int_{Q}\left|g(x)-P_{Q}^{\delta}(g)(x)\right|^{r}dx\right)^{\frac{1}{r}}\\
&=&|Q|\left(\frac{1}{|Q|}\int_{Q}\left|g(x)-P_{Q}^{\delta}(g)(x)\right|^{r}dx\right)^{\frac{1}{r}}\\
&=&\left\|\chi_{Q}\right\|_{q,p}\left[\frac{1}{\phi_{1}(Q)}\left(\frac{1}{|Q|}\int_{Q}\left|g(x)-P_{Q}^{\delta}(g)(x)\right|^{r}dx\right)^{\frac{1}{r}}\right].
\end{eqnarray*}
Consequently,  
\begin{eqnarray*}
\left\|g\right\|_{\mathcal{L}_{r,\phi_{1},\delta}}&:=&\sup_{Q\in\mathcal{Q}}\frac{1}{\phi_{1}(Q)}\left(\frac{1}{|Q|}\int_{Q}\left|g(x)-P_Q^{\delta}(g)(x)\right|^rdx\right)^{\frac{1}{r}}\\
&=&\sup_{\underset{Q\text{ open}}{Q\in\mathcal{Q}}}\frac{1}{\phi_{1}(Q)}\left(\frac{1}{|Q|}\int_{Q}\left|g(x)-P_{Q}^{\delta}(g)(x)\right|^{r}dx\right)^{\frac{1}{r}}\leq\sup_{\underset{Q \text{ open}}{Q\in\mathcal{Q}}}\frac{\textit{O}(g,Q,r)}{\left\|\chi_{Q}\right\|_{q,p}}\cdot\ \ (\circ)\\
\end{eqnarray*}

\begin{defn}
Suppose that $0<q\leq1$ and $0<p<+\infty$. Let $0<\eta<+\infty$ and $1\leq r<+\infty$. We say that a function $g$ in $L_{\mathrm{loc}}^r$ belongs to $\mathcal{L}_{r,\phi_{1},\delta}^{(q,p,\eta)}:=\mathcal{L}_{r,\phi_{1},\delta}^{(q,p,\eta)}(\mathbb{R}^d)$ if there is a constante $C>0$ such that, for all families of open subsets $\left\{\Omega^j\right\}_{j\in\mathbb{Z}}$ with $\left\|\sum_{j\in\mathbb{Z}}2^{j\eta}\chi_{\Omega^j}\right\|_{\frac{q}{\eta},\frac{p}{\eta}}<+\infty$, we have 
\begin{eqnarray}
\sum_{j\in\mathbb{Z}}2^j\textit{O}(g,\Omega^j,r)\leq C\left\|\sum_{j\in\mathbb{Z}}2^{j\eta}\chi_{\Omega^j}\right\|_{\frac{q}{\eta},\frac{p}{\eta}}^{\frac{1}{\eta}}. \label{dualqp}
\end{eqnarray} 
\end{defn}

It is easy to verify that $\mathcal{L}_{r,\phi_{1},\delta}^{(q,p,\eta)}$ contains $\mathcal{P_{\delta}}:=\mathcal{P_{\delta}}(\mathbb{R}^d)$. Moreover, when $g\in\mathcal{L}_{r,\phi_{1},\delta}^{(q,p,\eta)}$ , we define $\left\|g\right\|_{\mathcal{L}_{r,\phi_{1},\delta}^{(q,p,\eta)}}:=\inf\left\{C>0:\ C \text{ satisfies } (\ref{dualqp})\right\}$. It is also easy to see that $\left\|\cdot\right\|_{\mathcal{L}_{r,\phi_{1},\delta}^{(q,p,\eta)}}$ defines a semi-norm on $\mathcal{L}_{r,\phi_{1},\delta}^{(q,p,\eta)}$ and a norm on $\mathcal{L}_{r,\phi_{1},\delta}^{(q,p,\eta)}/\mathcal{P_{\delta}}$. So, we consider elements in $\mathcal{L}_{r,\phi_{1},\delta}^{(q,p,\eta)}$ modulo polynomials of degree $\delta$ so that $\mathcal{L}_{r,\phi_{1},\delta}^{(q,p,\eta)}$ is a Banach space. Thus, when we write $g\in\mathcal{L}_{r,\phi_{1},\delta}^{(q,p,\eta)}$, then $g$ stands for the representative of $\left\{g+\mathfrak{q}:\ \mathfrak{q}\ \text{ is polynomial of degree}\ \delta\right\}$.

We denote by $\mathcal{O}_{\mathbb{Z},q,p,\eta}$ the set of all families of open subsets $\left\{\Omega^j\right\}_{j\in\mathbb{Z}}$ such that $\left\|\sum_{j\in\mathbb{Z}}2^{j\eta}\chi_{\Omega^j}\right\|_{\frac{q}{\eta},\frac{p}{\eta}}<+\infty$. From the definition of $\left\|g\right\|_{\mathcal{L}_{r,\phi_{1},\delta}^{(q,p,\eta)}}$ , it follows that, for all families of open subsets $\left\{\Omega^j\right\}_{j\in\mathbb{Z}}$ of $\mathcal{O}_{\mathbb{Z},q,p,\eta}$, 
\begin{eqnarray}
\sum_{j\in\mathbb{Z}}2^j\textit{O}(g,\Omega^j,r)\leq\left\|g\right\|_{\mathcal{L}_{r,\phi_{1},\delta}^{(q,p,\eta)}}\left\|\sum_{j\in\mathbb{Z}}2^{j\eta}\chi_{\Omega^j}\right\|_{\frac{q}{\eta},\frac{p}{\eta}}^{\frac{1}{\eta}}. \label{dualqp2}
\end{eqnarray}
Thus, if $g\in\mathcal{L}_{r,\phi_{1},\delta}^{(q,p,\eta)}$ , with $0<q\leq1$ and $q\leq p<+\infty$, then 
\begin{eqnarray}
\left\|g\right\|_{\mathcal{L}_{r,\phi_{2},\delta}}\leq\left\|g\right\|_{\mathcal{L}_{r,\phi_{1},\delta}}\leq\left\|g\right\|_{\mathcal{L}_{r,\phi_{1},\delta}^{(q,p,\eta)}}, \label{dualqp3}
\end{eqnarray}
and so $\mathcal{L}_{r,\phi_{1},\delta}^{(q,p,\eta)}\hookrightarrow\mathcal{L}_{r,\phi_{1},\delta}\hookrightarrow\mathcal{L}_{r,\phi_{2},\delta}$. Indeed, we have   
\begin{eqnarray*}
\left\|g\right\|_{\mathcal{L}_{r,\phi_{2},\delta}}\leq\left\|g\right\|_{\mathcal{L}_{r,\phi_{1},\delta}}&\leq&\sup_{\underset{Q \text{ opent}}{Q\in\mathcal{Q}}}\frac{\textit{O}(g,Q,r)}{\left\|\chi_{Q}\right\|_{q,p}}\leq\left\|g\right\|_{\mathcal{L}_{r,\phi_{1},\delta}^{(q,p,\eta)}},
\end{eqnarray*}
by $(\circ)$ and the fact that, for any open cube $Q$, we have $\textit{O}(g,Q,r)\leq\left\|g\right\|_{\mathcal{L}_{r,\phi_{1},\delta}^{(q,p,\eta)}}\left\|\chi_{Q}\right\|_{q,p}$. 

We point out the inequality on the right in (\ref{dualqp3}) holds for $0<q\leq1$ and $0<p<+\infty$; the assumption $q\leq p<+\infty$ is only used for the inequality on the left.  

Furthermore, whenever $0<q,p\leq1$, and $0<\eta\leq1$, we have  
\begin{eqnarray}
\left\|g\right\|_{\mathcal{L}_{r,\phi_{1},\delta}^{(q,p,\eta)}}\leq C(d,q,p)\left\|g\right\|_{\mathcal{L}_{r,\phi_{1},\delta}}, \label{dualqp3bis}
\end{eqnarray}
for any $g\in\mathcal{L}_{r,\phi_{1},\delta}$, where $C(d,q,p)>0$ is a constant independent of $\eta$. Hence 
\begin{eqnarray}
\left\|g\right\|_{\mathcal{L}_{r,\phi_{1},\delta}^{(q,p,\eta)}}\approx\left\|g\right\|_{\mathcal{L}_{r,\phi_{1},\delta}}, \label{dualqp3bisbis}
\end{eqnarray}
whenever $0<q,p\leq1$ and $0<\eta\leq1$, by (\ref{dualqp3}) and (\ref{dualqp3bis}). Thus, $\mathcal{L}_{r,\phi_{1},\delta}^{(q,p,\eta)}=\mathcal{L}_{r,\phi_{1},\delta}$ , with equivalent norms, whenever $0<q,p\leq1$ and $0<\eta\leq1$. 

For the proof of (\ref{dualqp3bis}), consider a family of open subsets $\left\{\Omega^j\right\}_{j\in\mathbb{Z}}$ of $\mathcal{O}_{\mathbb{Z},q,p,\eta}$. Then, for all families of cubes $\left\{Q^{j,n}\right\}_{n\geq0}$ such that $Q^{j,n}\subset \Omega^j$, for all $n\geq0$, and $\sum_{n\geq0}\chi_{_{Q^{j,n}}}\leq K(d)$, with $\widetilde{Q^{j,n}}:=C_0Q^{j,n}$, we have    
\begin{eqnarray*}
&&\sum_{n\geq0}|\widetilde{Q^{j,n}}|^{\frac{1}{r'}}\left(\int_{\widetilde{Q^{j,n}}}\left|g(x)-P_{\widetilde{Q^{j,n}}}^{\delta}(g)(x)\right|^{r}dx\right)^{\frac{1}{r}}\\
&=&\sum_{n\geq0}\left\|\chi_{\widetilde{Q^{j,n}}}\right\|_{q,p}\left[\frac{|\widetilde{Q^{j,n}}|}{\left\|\chi_{\widetilde{Q^{j,n}}}\right\|_{q,p}}\left(\frac{1}{|\widetilde{Q^{j,n}}|}\int_{\widetilde{Q^{j,n}}}\left|g(x)-P_{\widetilde{Q^{j,n}}}^{\delta}(g)(x)\right|^{r}dx\right)^{\frac{1}{r}}\right]\\
&\leq&\left\|g\right\|_{\mathcal{L}_{r,\phi_{1},\delta}}\sum_{n\geq0}\left\|\chi_{\widetilde{Q^{j,n}}}\right\|_{q,p}\leq C(d,q,p)\left\|g\right\|_{\mathcal{L}_{r,\phi_{1},\delta}}\sum_{n\geq0}\left\|\chi_{Q^{j,n}}\right\|_{q,p}\\
&\leq&C(d,q,p)\left\|g\right\|_{\mathcal{L}_{r,\phi_{1},\delta}}\left\|\sum_{n\geq0}\chi_{Q^{j,n}}\right\|_{q,p}\leq C(d,q,p)\left\|g\right\|_{\mathcal{L}_{r,\phi_{1},\delta}}\left\|\chi_{\Omega^j}\right\|_{q,p},
\end{eqnarray*}
by Proposition \ref{InverseHoldMink3} and the fact that $\sum_{n\geq0}\chi_{_{Q^{j,n}}}\leq K(d)\chi_{\Omega^j}$. Therefore,  
\begin{eqnarray*}
\sum_{j\in\mathbb{Z}}2^j\textit{O}(g,\Omega^j,r)&\leq&C(d,q,p)\left\|g\right\|_{\mathcal{L}_{r,\phi_{1},\delta}}\sum_{j\in\mathbb{Z}}2^j\left\|\chi_{\Omega^j}\right\|_{q,p}\\
&\leq&C(d,q,p)\left\|g\right\|_{\mathcal{L}_{r,\phi_{1},\delta}}\left\|\sum_{j\in\mathbb{Z}}2^j\chi_{\Omega^j}\right\|_{q,p}\\
&\leq&C(d,q,p)\left\|g\right\|_{\mathcal{L}_{r,\phi_{1},\delta}}\left\|\sum_{j\in\mathbb{Z}}2^{j\eta}\chi_{\Omega^j}\right\|_{\frac{q}{\eta},\frac{p}{\eta}}^{\frac{1}{\eta}}.
\end{eqnarray*}
So, $g\in\mathcal{L}_{r,\phi_{1},\delta}^{(q,p,\eta)}$ and $\left\|g\right\|_{\mathcal{L}_{r,\phi_{1},\delta}^{(q,p,\eta)}}\leq C(d,q,p)\left\|g\right\|_{\mathcal{L}_{r,\phi_{1},\delta}}$; which completes the proof of (\ref{dualqp3bis}).\\

We can also define a norm $|||\cdot|||_{\mathcal{L}_{r,\phi_{1},\delta}^{(q,p,\eta)}}$ equivalent to $\left\|\cdot\right\|_{\mathcal{L}_{r,\phi_{1},\delta}^{(q,p,\eta)}}$ on $\mathcal{L}_{r,\phi_1,\delta}^{(q,p,\eta)}$, with $0<q\leq1$, $0<p<+\infty$, $0<\eta<+\infty$ and $1\leq r<+\infty$. Indeed, given a function $g$ in $L_{\mathrm{loc}}^r$ and an open subset $\Omega$ such that $\Omega\neq\mathbb{R}^d$, we put $$\widetilde{\textit{O}(g,\Omega,r)}:=\sup{\sum_{n\geq0}\inf_{\mathfrak{p}\in\mathcal{P_{\delta}}}|\widetilde{Q^n}|^{\frac{1}{r'}}\left(\int_{\widetilde{Q^n}}\left|g(x)-\mathfrak{p}(x)\right|^{r}dx\right)^{\frac{1}{r}}},$$ where $\frac{1}{r}+\frac{1}{r'}=1$ and the supremum is taken over all families of cubes $\left\{Q^n\right\}_{n\geq0}$ such that $Q^n\subset\Omega$, for all $n\geq0$ and $\sum_{n\geq0}\chi_{_{Q^n}}\leq K(d)$, with $\widetilde{Q^n}=C_0Q^n$, where $K(d)>0$ and $C_0>1$ are the constants of the definition of $\textit{O}(g,\Omega,r)$. We have the following proposition.  

\begin{prop}\label{dualqpequival0}
Let $g$ be a function in $L_{\mathrm{loc}}^r$. If there is a constant $C>0$ such that, for all families of open subsets $\left\{\Omega^j\right\}_{j\in\mathbb{Z}}$ with $\left\|\sum_{j\in\mathbb{Z}}2^{j\eta}\chi_{\Omega^j}\right\|_{\frac{q}{\eta},\frac{p}{\eta}}<+\infty$, we have  
\begin{eqnarray}
\sum_{j\in\mathbb{Z}}2^j\widetilde{\textit{O}(g,\Omega^j,r)}\leq C\left\|\sum_{j\in\mathbb{Z}}2^{j\eta}\chi_{\Omega^j}\right\|_{\frac{q}{\eta},\frac{p}{\eta}}^{\frac{1}{\eta}}, \label{dualqpequival}
\end{eqnarray} 
then $g\in\mathcal{L}_{r,\phi_{1},\delta}^{(q,p,\eta)}$. Moreover, if we define $$|||g|||_{\mathcal{L}_{r,\phi_{1},\delta}^{(q,p,\eta)}}:=\inf\left\{C>0:\ C \text{ satisfies } (\ref{dualqpequival})\right\},$$ then   
\begin{eqnarray}
|||g|||_{\mathcal{L}_{r,\phi_{1},\delta}^{(q,p,\eta)}}\approx\left\|g\right\|_{\mathcal{L}_{r,\phi_{1},\delta}^{(q,p,\eta)}}. \label{dualqpequival1}
\end{eqnarray} 
\end{prop}
\begin{proof} 
Let $\left\{\Omega^j\right\}_{j\in\mathbb{Z}}$ be a family of open subsets fulfilling the hypotheses of the proposition. First, we show that, for any open subset $\Omega$ such that $\Omega\neq\mathbb{R}^d$, we have  
\begin{eqnarray}
\widetilde{\textit{O}(g,\Omega,r)}\leq\textit{O}(g,\Omega,r)\leq (1+C)\widetilde{\textit{O}(g,\Omega,r)}, \label{dualqpequival2}
\end{eqnarray}
where the constant $C>0$ is the one of Remark \ref{0laremarquetile0}. By definition of $\widetilde{\textit{O}(g,\Omega,r)}$ and $\textit{O}(g,\Omega,r)$, clearly, we have $$\widetilde{\textit{O}(g,\Omega,r)}\leq\textit{O}(g,\Omega,r).$$ For the inequality on the right in (\ref{dualqpequival2}), we consider a family of cubes $\left\{Q^n\right\}_{n\geq0}$ such that $Q^n\subset\Omega$, for all $n\geq0$ and $\sum_{n\geq0}\chi_{_{Q^n}}\leq K(d)$, with $\widetilde{Q^n}=C_0Q^n$. Let $\mathfrak{p}\in\mathcal{P_{\delta}}$. We have  
\begin{eqnarray*}
\left(\int_{\widetilde{Q^n}}\left|g(x)-P_{\widetilde{Q^n}}^{\delta}(g)(x)\right|^{r}dx\right)^{\frac{1}{r}}\leq(1+C)\left(\int_{\widetilde{Q^n}}\left|g(x)-\mathfrak{p}(x)\right|^{r}dx\right)^{\frac{1}{r}},
\end{eqnarray*}
where the constant $C>0$ is the one of Remark \ref{0laremarquetile0}. The proof of this inequality is standard (see, for instance, \cite{MBOW}, p. 53). Hence  
\begin{eqnarray*}
&&\textit{O}(g,\Omega,r)=\sup{\sum_{n\geq0}|\widetilde{Q^n}|^{\frac{1}{r'}}\left(\int_{\widetilde{Q^n}}\left|g(x)-P_{\widetilde{Q^n}}^{\delta}(g)(x)\right|^{r}dx\right)^{\frac{1}{r}}}\\
&\leq&(1+C)\sup{\sum_{n\geq0}\inf_{\mathfrak{p}\in\mathcal{P_{\delta}}}|\widetilde{Q^n}|^{\frac{1}{r'}}\left(\int_{\widetilde{Q^n}}\left|g(x)-\mathfrak{p}(x)\right|^{r}dx\right)^{\frac{1}{r}}}=(1+C)\widetilde{\textit{O}(g,\Omega,r)},
\end{eqnarray*}
which completes the proof of (\ref{dualqpequival2}). From (\ref{dualqpequival}) and (\ref{dualqpequival2}), it follows that $$\sum_{j\in\mathbb{Z}}2^j\textit{O}(g,\Omega^j,r)\leq C(1+C)\left\|\sum_{j\in\mathbb{Z}}2^{j\eta}\chi_{\Omega^j}\right\|_{\frac{q}{\eta},\frac{p}{\eta}}^{\frac{1}{\eta}}.$$ Therefore, $g\in\mathcal{L}_{r,\phi_{1},\delta}^{(q,p,\eta)}$, with $\left\|g\right\|_{\mathcal{L}_{r,\phi_{1},\delta}^{(q,p,\eta)}}\leq(1+C)|||g|||_{\mathcal{L}_{r,\phi_{1},\delta}^{(q,p,\eta)}}$. Furthermore, since $g\in\mathcal{L}_{r,\phi_{1},\delta}^{(q,p,\eta)}$, we have, by (\ref{dualqpequival2}) and (\ref{dualqp2}), $$\sum_{j\in\mathbb{Z}}2^j\widetilde{\textit{O}(g,\Omega^j,r)}\leq\sum_{j\in\mathbb{Z}}2^j\textit{O}(g,\Omega^j,r)\leq\left\|g\right\|_{\mathcal{L}_{r,\phi_{1},\delta}^{(q,p,\eta)}}\left\|\sum_{j\in\mathbb{Z}}2^{j\eta}\chi_{\Omega^j}\right\|_{\frac{q}{\eta},\frac{p}{\eta}}^{\frac{1}{\eta}}.$$ Hence $|||g|||_{\mathcal{L}_{r,\phi_{1},\delta}^{(q,p,\eta)}}\leq\left\|g\right\|_{\mathcal{L}_{r,\phi_{1},\delta}^{(q,p,\eta)}}$. This states (\ref{dualqpequival1}).
\end{proof}

Moreover, one easily verifies that $|||\cdot|||_{\mathcal{L}_{r,\phi_{1},\delta}^{(q,p,\eta)}}$ is a norm on $\mathcal{L}_{r,\phi_1,\delta}^{(q,p,\eta)}$.\\   

Our duality result for $\mathcal{H}^{(q,p)}$, with $0<q\leq 1<p<+\infty$, can be stated as follows. 

\begin{thm} \label{theoremdualqp}
Suppose that $0<q\leq1<p<+\infty$. Let $p<r\leq+\infty$. Then, $\left(\mathcal{H}^{(q,p)}\right)^{\ast}$ is isomorphic to $\mathcal{L}_{r',\phi_1,\delta}^{(q,p,\eta)}$ , where $\frac{1}{r}+\frac{1}{r'}=1$, $0<\eta<q$ if $r<+\infty$, and $0<\eta\leq1$ if $r=+\infty$, with equivalent norms. More precisely, we have the following assertions:
\begin{enumerate}
\item Let $g\in\mathcal{L}_{r',\phi_1,\delta}^{(q,p,\eta)}$. Then, the mapping $$T_g:\mathcal{H}_{fin}^{(q,p)}\ni f\longmapsto\int_{\mathbb{R}^d}g(x)f(x)dx,$$ where $\mathcal{H}_{fin}^{(q,p)}$ is the subspace of $\mathcal{H}^{(q,p)}$ consisting of finite linear combinations of $(q,r,\delta)$-atoms, extends to a unique continuous linear functional $\widetilde{T_g}$ on $\mathcal{H}^{(q,p)}$ such that 
\begin{eqnarray*}
\left\|\widetilde{T_g}\right\|=\left\|T_g\right\|\leq C\left\|g\right\|_{\mathcal{L}_{r',\phi_{1},\delta}^{(q,p,\eta)}}, 
\end{eqnarray*} 
where $C>0$ is a constant independent of $g$. \label{dualpointqp1}
\item Conversely, for any $T\in\left(\mathcal{H}^{(q,p)}\right)^{\ast}$, there exists $g\in\mathcal{L}_{r',\phi_1,\delta}^{(q,p,\eta)}$ such that $T=T_g$; namely $$T(f)=\int_{\mathbb{R}^d}g(x)f(x)dx,\ \text{ for all }\ f\in\mathcal{H}_{fin}^{(q,p)},$$ and
\begin{eqnarray*}
\left\|g\right\|_{\mathcal{L}_{r',\phi_{1},\delta}^{(q,p,\eta)}}\leq C\left\|T\right\|, 
\end{eqnarray*} 
where $C>0$ is a constant independent of $T$. \label{dualpointqp2}
\end{enumerate}
\end{thm}
\begin{proof}
Let $g\in\mathcal{L}_{r',\phi_1,\delta}^{(q,p,\eta)}$. Let $f\in\mathcal{H}_{fin}^{(q,p)}$, where $\mathcal{H}_{fin}^{(q,p)}$ is the subspace of $\mathcal{H}^{(q,p)}$ consisting of finite linear combinations of $(q,r,\delta)$-atoms. Then, $f=\sum_{j=-\infty}^{+\infty}\sum_{k\geq0}\lambda_{j,k}a_{j,k}$ almost everywhere and in the sense of $\mathcal{S'}$, with $\lambda_{j,k}=C_{1}2^j|\widetilde{Q^{j,k}}|^{\frac{1}{q}}$, $\left\|a_{j,k}\right\|_{L^\infty}\leq|\widetilde{Q^{j,k}}|^{-\frac{1}{q}}$, $\left(a_{j,k},\widetilde{Q^{j,k}}\right)\in\mathcal{A}(q,\infty,\delta)$, $\widetilde{Q^{j,k}}:=C_0Q^{j,k}$, $Q^{j,k}\subset\mathcal{O}^j$ ($\mathcal{O}^j$ being a certain open subset not equal to $\mathbb{R}^d$), for all $k\geq0$, $\sum_{k\geq0}\chi_{_{Q^{j,k}}}\leq K(d)$ and 
\begin{eqnarray*}
\left\|\sum_{j=-\infty}^{+\infty}\sum_{k\geq 0}\left(\frac{|\lambda_{j,k}|}{\left\|\chi_{_{\widetilde{Q^{j,k}}}}\right\|_{q}}\right)^{\eta}\chi_{_{\widetilde{Q^{j,k}}}}\right\|_{\frac{q}{\eta},\frac{p}{\eta}}^{\frac{1}{{\eta}}}\approx \left\|\sum_{j=-\infty}^{+\infty}2^{j\eta}\chi_{\mathcal{O}^j}\right\|_{\frac{q}{\eta},\frac{p}{\eta}}^{\frac{1}{{\eta}}}\approx\left\|f\right\|_{\mathcal{H}^{(q,p)}},
\end{eqnarray*}
$0<\eta\leq1$, by the proof of Theorem 4.4 in \cite{AbFt}, and Theorem 4.3 in \cite{AbFt}. Moreover, since $f=\sum_{j=-\infty}^{+\infty}\sum_{k\geq0}\lambda_{j,k}a_{j,k}$ also (inconditionally) in $\mathcal{H}^q$ (because $\mathcal{H}_{fin}^{(q,p)}\subset\mathcal{H}^q$), each $\left(a_{j,k},\widetilde{Q^{j,k}}\right)$ is also in $\mathcal{A}(q,r,\delta)$, and $\mathcal{L}_{r',\phi_1,\delta}^{(q,p,\eta)}\subset\mathcal{L}_{r',\phi_2,\delta}$ (see (\ref{dualqp3})) with $\mathcal{L}_{r',\phi_2,\delta}$ being the topological dual of $\mathcal{H}^q$, we have $$\int_{\mathbb{R}^d}g(x)f(x)dx=\sum_{j=-\infty}^{+\infty}\sum_{k\geq0}\lambda_{j,k}\int_{\mathbb{R}^d}g(x)a_{j,k}(x)dx$$ (see \cite{MBOW} or \cite{JGLR}, for details). Hence 
\begin{eqnarray*}
|T_g(f)|&=&\left|\int_{\mathbb{R}^d}g(x)f(x)dx\right|\leq\sum_{j=-\infty}^{+\infty}\sum_{k\geq0}|\lambda_{j,k}|\left|\int_{\mathbb{R}^d}g(x)a_{j,k}(x)dx\right|\\
&=&C_{1}\sum_{j=-\infty}^{+\infty}\sum_{k\geq0}2^j|\widetilde{Q^{j,k}}|^{\frac{1}{q}}\left|\int_{\widetilde{Q^{j,k}}}(g(x)-P_{\widetilde{Q^{j,k}}}^\delta(g)(x))a_{j,k}(x)dx\right|\\
&\leq&C_{1}\sum_{j=-\infty}^{+\infty}\sum_{k\geq0}2^j|\widetilde{Q^{j,k}}|^{\frac{1}{q}}\left\|a_{j,k}\right\|_{r}\left(\int_{\widetilde{Q^{j,k}}}|g(x)-P_{\widetilde{Q^{j,k}}}^\delta(g)(x)|^{r'}dx\right)^{\frac{1}{r'}}\\
&\leq&C_{1}\sum_{j=-\infty}^{+\infty}2^j\sum_{k\geq0}|\widetilde{Q^{j,k}}|^{\frac{1}{r}}\left(\int_{\widetilde{Q^{j,k}}}|g(x)-P_{\widetilde{Q^{j,k}}}^\delta(g)(x)|^{r'}dx\right)^{\frac{1}{r'}}\\
&\leq&C_{1}\sum_{j=-\infty}^{+\infty}2^j\textit{O}(g,\mathcal{O}^j,r')\\
&\leq&C_{1}\left\|g\right\|_{\mathcal{L}_{r',\phi_{1},\delta}^{(q,p,\eta)}}\left\|\sum_{j=-\infty}^{+\infty}2^{j\eta}\chi_{\mathcal{O}^j}\right\|_{\frac{q}{\eta},\frac{p}{\eta}}^{\frac{1}{\eta}}\leq C\left\|g\right\|_{\mathcal{L}_{r',\phi_{1},\delta}^{(q,p,\eta)}}\left\|f\right\|_{\mathcal{H}^{(q,p)}}. 
\end{eqnarray*}  
This shows that $T_g$ extends to a unique continuous linear functional $\widetilde{T_g}$ on $\mathcal{H}^{(q,p)}$, with $\left\|\widetilde{T_g}\right\|=\left\|T_g\right\|\leq C\left\|g\right\|_{\mathcal{L}_{r',\phi_{1},\delta}^{(q,p,\eta)}}$, where $C>0$ is a constant independent of $g$, given that $\mathcal{H}_{fin}^{(q,p)}$ is a dense subspace of $\mathcal{H}^{(q,p)}$ with respect to the quasi-norm $\left\|\cdot\right\|_{\mathcal{H}^{(q,p)}}$.\\

Now we prove the second assertion. Let $T\in\left(\mathcal{H}^{(q,p)}\right)^{\ast}$. Then, there exists $g\in\mathcal{L}_{r',\phi_1,\delta}$ such that $T=T_g$; namely $$T(f)=T_g(f)=\int_{\mathbb{R}^d}g(x)f(x)dx,\ \text{ for all }\ f\in\mathcal{H}_{fin}^{(q,p)},\ \ (\otimes)$$ by the proof of Theorem \ref{theoremdual} (\cite{AbFt1}, Theorem 3.3). We are going to show that $g$ is in $\mathcal{L}_{r',\phi_1,\delta}^{(q,p,\eta)}$. Let $\left\{\Omega^j\right\}_{j\in\mathbb{Z}}$ be a family of open subsets such that $\left\|\sum_{j\in\mathbb{Z}}2^{j\eta}\chi_{\Omega^j}\right\|_{\frac{q}{\eta},\frac{p}{\eta}}<+\infty$, and $\left\{Q^{j,k}\right\}_{k\geq0}$ be a family of cubes such that $Q^{j,k}\subset\Omega^j$, for all $k\geq0$, $\sum_{k\geq0}\chi_{_{Q^{j,k}}}\leq K(d)$, with $\widetilde{Q^{j,k}}:=C_0Q^{j,k}$, and $$\textit{O}(g,\Omega^j,r')\leq 2\sum_{k\geq0}|\widetilde{Q^{j,k}}|^{\frac{1}{r}}\left(\int_{\widetilde{Q^{j,k}}}|g(x)-P_{\widetilde{Q^{j,k}}}^\delta(g)(x)|^{r'}dx\right)^{\frac{1}{r'}}.\ \ (\otimes\otimes)$$ 

We note that the hypothesis $(\otimes\otimes)$ is justified. In fact, if $\textit{O}(g,\Omega^j,r')=0$, then $(\otimes\otimes)$ is evident, since for any family of cubes $\left\{Q^{j,k}\right\}_{k\geq0}$, the member on the right is null, by the definition of $\textit{O}(g,\Omega^j,r')$. If $0<\textit{O}(g,\Omega^j,r')<+\infty$, then by the definition of $\textit{O}(g,\Omega^j,r')$, it is not hard to find a family of cubes $\left\{Q^{j,k}\right\}_{k\geq0}$ verifying the hypothesis $(\otimes\otimes)$. Then, the justification of $(\otimes\otimes)$ amounts to prove that 
\begin{eqnarray}
\textit{O}(g,\Omega^j,r')<+\infty. \label{apresrefexqp}
\end{eqnarray}
The proof of (\ref{apresrefexqp}) being contained in the sequel, we admit it for the moment, and so $(\otimes\otimes)$ also. 

Let $h_{j,k}\in L^{r}(\widetilde{Q^{j,k}})$ with $\left\|h_{j,k}\right\|_{L^r(\widetilde{Q^{j,k}})}=1$ such that $$\left(\int_{\widetilde{Q^{j,k}}}\left|g(x)-P_{\widetilde{Q^{j,k}}}^{\delta}(g)(x)\right|^{r'}dx\right)^{\frac{1}{r'}}=\int_{\widetilde{Q^{j,k}}}\left(g(x)-P_{\widetilde{Q^{j,k}}}^{\delta}(g)(x)\right)h_{j,k}(x)dx.$$ Put $$a_{j,k}(x):=C_{h_{j,k},Q^{j,k},\delta}|\widetilde{Q^{j,k}}|^{\frac{1}{r}-\frac{1}{q}}\left(h_{j,k}(x)-P_{\widetilde{Q^{j,k}}}^{\delta}(h_{j,k})(x)\right)\chi_{\widetilde{Q^{j,k}}}(x),$$ with $$C_{h_{j,k},Q^{j,k},\delta}:=\left(1+\left\|\left(h_{j,k}-P_{\widetilde{Q^{j,k}}}^{\delta}(h_{j,k})\right)\chi_{\widetilde{Q^{j,k}}}\right\|_r\right)^{-1}.$$ We have $(a_{j,k},\widetilde{Q^{j,k}})\in\mathcal{A}(q,r,\delta)$ and $1+\left\|\left(h_{j,k}-P_{\widetilde{Q^{j,k}}}^{\delta}(h_{j,k})\right)\chi_{\widetilde{Q^{j,k}}}\right\|_r=C_{h_{j,k},Q^{j,k},\delta}^{-1}\leq2+C$, where $C>0$ is a constant independent of $h_{j,k}$ and $\widetilde{Q^{j,k}}$, by Remark \ref{0laremarquetile0} (\ref{3Campanato3}) and the fact that $\left\|h_{j,k}\right\|_{L^r(\widetilde{Q^{j,k}})}=1$. Put $\lambda_{j,k}:=2^j|\widetilde{Q^{j,k}}|^{\frac{1}{q}}$. Consider a real $\gamma>\max\left\{1,\frac{\eta}{q}\right\}$, for instance $\gamma=1+\frac{\eta}{q}\cdot$ We have $1<\gamma$ and $1<\frac{q\gamma}{\eta}\leq\frac{p\gamma}{\eta}\cdot$ Thus, with Proposition \ref{operamaxima}, by arguing as in the proof of Theorem 4.4 in \cite{AbFt} (see p. 1919), we obtain  
\begin{eqnarray*}
\left\|\sum_{j\in\mathbb{Z}}\sum_{k\geq 0}\left(\frac{|\lambda_{j,k}|}{\left\|\chi_{_{\widetilde{Q^{j,k}}}}\right\|_{q}}\right)^{\eta}\chi_{_{\widetilde{Q^{j,k}}}}\right\|_{\frac{q}{\eta},\frac{p}{\eta}}^{\frac{1}{{\eta}}}
&\leq&C(d,\gamma,\eta)\left\|\sum_{j\in\mathbb{Z}}\sum_{k\geq 0}2^{j\eta}\left[\mathfrak{M}\left(\chi_{_{Q^{j,k}}}\right)\right]^{\gamma}\right\|_{\frac{q}{\eta},\frac{p}{\eta}}^{\frac{1}{{\eta}}}\\ 
&\leq&C(q,p,d,\gamma,\eta)\left\|\sum_{j\in\mathbb{Z}}\sum_{k\geq 0}2^{j\eta}\chi_{_{Q^{j,k}}}\right\|_{\frac{q}{\eta},\frac{p}{\eta}}^{\frac{1}{\eta}}\\
&\leq&C(q,p,d,\eta)\left\|\sum_{j\in\mathbb{Z}}2^{j\eta}\chi_{\Omega^j}\right\|_{\frac{q}{\eta},\frac{p}{\eta}}^{\frac{1}{\eta}}<+\infty.
\end{eqnarray*} 
Then, considering that $0<\eta<q$ (if $p<r<+\infty$) or $0<\eta\leq 1$ (if $r=+\infty$), we have $\sum_{j\in\mathbb{Z}}\sum_{k\geq 0}\lambda_{j,k}a_{j,k}\in\mathcal{H}^{(q,p)}$ and 
\begin{eqnarray*}
\left\|\sum_{j\in\mathbb{Z}}\sum_{k\geq 0}\lambda_{j,k}a_{j,k}\right\|_{\mathcal{H}^{(q,p)}}&\leq&C(\varphi,q,p,d,\delta,r,\eta)\left\|\sum_{j\in\mathbb{Z}}\sum_{k\geq 0}\left(\frac{|\lambda_{j,k}|}{\left\|\chi_{_{\widetilde{Q^{j,k}}}}\right\|_{q}}\right)^{\eta}\chi_{_{\widetilde{Q^{j,k}}}}\right\|_{\frac{q}{\eta},\frac{p}{\eta}}^{\frac{1}{{\eta}}}\\
&\leq&C(\varphi,q,p,d,\delta,r,\eta)\left\|\sum_{j\in\mathbb{Z}}2^{j\eta}\chi_{\Omega^j}\right\|_{\frac{q}{\eta},\frac{p}{\eta}}^{\frac{1}{\eta}},\ \ (\otimes\otimes\otimes)
\end{eqnarray*}
by Theorem 4.6 in \cite{AbFt} (if $p<r<+\infty$) or Theorem 4.3 in \cite{AbFt} (if $r=+\infty$). 

Also, by $(\otimes\otimes)$ and (\ref{Campanato1}), we have 
\begin{eqnarray*}
\sum_{j\in\mathbb{Z}}2^j\textit{O}(g,\Omega^j,r')&\leq&2\sum_{j\in\mathbb{Z}}2^j\sum_{k\geq0}|\widetilde{Q^{j,k}}|^{\frac{1}{r}}\left(\int_{\widetilde{Q^{j,k}}}|g(x)-P_{\widetilde{Q^{j,k}}}^\delta(g)(x)|^{r'}dx\right)^{\frac{1}{r'}}\\
&=&2\sum_{j\in\mathbb{Z}}2^j\sum_{k\geq0}|\widetilde{Q^{j,k}}|^{\frac{1}{r}}\int_{\widetilde{Q^{j,k}}}(g(x)-P_{\widetilde{Q^{j,k}}}^{\delta}(g)(x))h_{j,k}(x)dx\\
&=&2\sum_{j\in\mathbb{Z}}2^j\sum_{k\geq0}|\widetilde{Q^{j,k}}|^{\frac{1}{r}}\int_{\widetilde{Q^{j,k}}}g(x)\left(h_{j,k}(x)-P_{\widetilde{Q^{j,k}}}^{\delta}(h_{j,k})(x)\right)dx\\
&=&2\sum_{j\in\mathbb{Z}}2^j\sum_{k\geq0}|\widetilde{Q^{j,k}}|^{\frac{1}{r}}\int_{\widetilde{Q^{j,k}}}g(x)\left(h_{j,k}(x)-P_{\widetilde{Q^{j,k}}}^{\delta}(h_{j,k})(x)\right)\chi_{\widetilde{Q^{j,k}}}(x)dx.
\end{eqnarray*}
Furthermore,  
\begin{eqnarray*}
\left(h_{j,k}(x)-P_{\widetilde{Q^{j,k}}}^{\delta}(h_{j,k})(x)\right)\chi_{\widetilde{Q^{j,k}}}(x)=C_{h_{j,k},Q^{j,k},\delta}^{-1}|\widetilde{Q^{j,k}}|^{\frac{1}{q}-\frac{1}{r}}a_{j,k}(x),
\end{eqnarray*}
by definition of $a_{j,k}$. Hence  
\begin{eqnarray*}
\sum_{j\in\mathbb{Z}}2^j\textit{O}(g,\Omega^j,r')&\leq&2\sum_{j\in\mathbb{Z}}2^j\sum_{k\geq0}|\widetilde{Q^{j,k}}|^{\frac{1}{r}}C_{h_{j,k},Q^{j,k},\delta}^{-1}|\widetilde{Q^{j,k}}|^{\frac{1}{q}-\frac{1}{r}}\int_{\widetilde{Q^{j,k}}}g(x)a_{j,k}(x)dx\\
&\leq&C\sum_{j\in\mathbb{Z}}\sum_{k\geq0}2^j|\widetilde{Q^{j,k}}|^{\frac{1}{q}}\int_{\widetilde{Q^{j,k}}}g(x)a_{j,k}(x)dx\\
&=&C\sum_{j\in\mathbb{Z}}\sum_{k\geq0}\lambda_{j,k}\int_{\widetilde{Q^{j,k}}}g(x)a_{j,k}(x)dx\\
&=&C\sum_{j\in\mathbb{Z}}\sum_{k\geq0}\lambda_{j,k}T(a_{j,k})=CT\left(\sum_{j\in\mathbb{Z}}\sum_{k\geq0}\lambda_{j,k}a_{j,k}\right)\\
&=&C\left|T\left(\sum_{j\in\mathbb{Z}}\sum_{k\geq0}\lambda_{j,k}a_{j,k}\right)\right|\leq C\left\|T\right\|\left\|\sum_{j\in\mathbb{Z}}\sum_{k\geq 0}\lambda_{j,k}a_{j,k}\right\|_{\mathcal{H}^{(q,p)}}\\
&\leq&C\left\|T\right\|\left\|\sum_{j\in\mathbb{Z}}2^{j\eta}\chi_{\Omega^j}\right\|_{\frac{q}{\eta},\frac{p}{\eta}}^{\frac{1}{\eta}},
\end{eqnarray*}
by $(\otimes)$ and $(\otimes\otimes\otimes)$. Therefore, $g\in\mathcal{L}_{r',\phi_1,\delta}^{(q,p,\eta)}$ and $\left\|g\right\|_{\mathcal{L}_{r',\phi_1,\delta}^{(q,p,\eta)}}\leq C\left\|T\right\|$.\\ 

The proof of Theorem \ref{theoremdualqp} will be so complete if we prove (\ref{apresrefexqp}). For this, we consider, for each open subset $\Omega^{j}$ of the family of open subsets $\left\{\Omega^j\right\}_{j\in\mathbb{Z}}$, a family of cubes $\left\{Q^{j,k}\right\}_{k\geq0}$ such that $Q^{j,k}\subset\Omega^{j}$, for all $k\geq0$, $\sum_{k\geq0}\chi_{_{Q^{j,k}}}\leq K(d)$, with $\widetilde{Q^{j,k}}:=C_0Q^{j,k}$. Next, we arbitrarily choose an open subset $\Omega^{j_0}$ in $\left\{\Omega^j\right\}_{j\in\mathbb{Z}}$. Then, arguing as above, we have 
\begin{eqnarray*}
&&2^{j_0}\sum_{k\geq0}|\widetilde{Q^{j_0,k}}|^{\frac{1}{r}}\left(\int_{\widetilde{Q^{j_0,k}}}|g(x)-P_{\widetilde{Q^{j_0,k}}}^\delta(g)(x)|^{r'}dx\right)^{\frac{1}{r'}}\\
&\leq&\sum_{j\in\mathbb{Z}}2^j\sum_{k\geq0}|\widetilde{Q^{j,k}}|^{\frac{1}{r}}\left(\int_{\widetilde{Q^{j,k}}}|g(x)-P_{\widetilde{Q^{j,k}}}^\delta(g)(x)|^{r'}dx\right)^{\frac{1}{r'}}\\
&\leq&C\left\|T\right\|\left\|\sum_{j\in\mathbb{Z}}2^{j\eta}\chi_{\Omega^j}\right\|_{\frac{q}{\eta},\frac{p}{\eta}}^{\frac{1}{\eta}}<+\infty.
\end{eqnarray*} 
Of course, for these above relations, we do not use $(\otimes\otimes)$, which must be justified. Hence  
\begin{eqnarray*}
2^{j_0}\textit{O}(g,\Omega^{j_0},r')&=&2^{j_0}\sup \sum_{k\geq0}|\widetilde{Q^{j_0,k}}|^{\frac{1}{r}}\left(\int_{\widetilde{Q^{j_0,k}}}|g(x)-P_{\widetilde{Q^{j_0,k}}}^\delta(g)(x)|^{r'}dx\right)^{\frac{1}{r'}}\\
&\leq&C\left\|T\right\|\left\|\sum_{j\in\mathbb{Z}}2^{j\eta}\chi_{\Omega^j}\right\|_{\frac{q}{\eta},\frac{p}{\eta}}^{\frac{1}{\eta}}<+\infty.
\end{eqnarray*}
We infer from it that $\textit{O}(g,\Omega^{j_0},r')<+\infty$, which proves (\ref{apresrefexqp}), and therefore $(\otimes\otimes)$. This completes the proof of Theorem \ref{theoremdualqp}.
\end{proof}

\begin{remark}\label{remarqedualeqp1}
Suppose that $0<q\leq1<p<+\infty$. Let $1<r<p'$ , $0<\eta<q$ and $0<\eta_1\leq1$. Then, we have $$\mathcal{L}_{1,\phi_1,\delta}^{(q,p,\eta_1)}\cong\left(\mathcal{H}^{(q,p)}\right)^{\ast}\cong\mathcal{L}_{r,\phi_1,\delta}^{(q,p,\eta)}\hookrightarrow\mathcal{L}_{r,\phi_2,\delta}\cong\left(\mathcal{H}^q\right)^{\ast}$$ and $$\mathcal{L}_{1,\phi_1,\delta}^{(q,p,\eta_1)}\cong\left(\mathcal{H}^{(q,p)}\right)^{\ast}\cong\mathcal{L}_{r,\phi_1,\delta}^{(q,p,\eta)}\hookrightarrow L^{p'}\cong\left(\mathcal{H}^p\right)^{\ast},$$ by Theorem \ref{theoremdualqp}, since $L^p\cong\mathcal{H}^p\hookrightarrow\mathcal{H}^{(q,p)}$. Thus, when $q=1$, we have $$(L^{\infty},\ell^{p'})\hookrightarrow\mathcal{L}_{1,\phi_1,\delta}^{(1,p,\eta_1)}\cong\mathcal{L}_{r,\phi_1,\delta}^{(1,p,\eta)}\hookrightarrow BMO(\mathbb{R}^d)$$ and $$(L^{\infty},\ell^{p'})\hookrightarrow\mathcal{L}_{1,\phi_1,\delta}^{(1,p,\eta_1)}\cong\mathcal{L}_{r,\phi_1,\delta}^{(1,p,\eta)}\hookrightarrow L^{p'},$$ since $\mathcal{H}^1\hookrightarrow\mathcal{H}^{(1,p)}\hookrightarrow(L^1,\ell^p)$ (see \cite{AbFt}, Theorem 3.2 ), $(L^1,\ell^p)^{\ast}\cong(L^{\infty},\ell^{p'})$ and $\mathcal{L}_{r,\phi_2,\delta}\cong(\mathcal{H}^1)^{\ast}\cong BMO(\mathbb{R}^d)$. Moreover, the inclusion of $(L^{\infty},\ell^{p'})$ in $\mathcal{L}_{1,\phi_1,\delta}^{(1,p,\eta_1)}\cong\mathcal{L}_{r,\phi_1,\delta}^{(1,p,\eta)}$ is strict, because $(L^{\infty},\ell^{p'})\subset L^{\infty}$, $\mathcal{P_{\delta}}\not\subset L^{\infty}$ and $\mathcal{P_{\delta}}\subset\mathcal{L}_{1,\phi_1,\delta}^{(1,p,\eta_1)}\cong\mathcal{L}_{r,\phi_1,\delta}^{(1,p,\eta)}$. Then, can one deduce from it that $\mathcal{H}^{(1,p)}\subsetneq(L^1,\ell^p)$? We will deal with this question in an upcoming article. 
\end{remark}

By combining Theorem \ref{theoremdual} with (\ref{dualqp3bisbis}), and Theorem \ref{theoremdualqp}, we obtain the following unified characterization of the topological dual of the Hardy-amalgam spaces $\mathcal{H}^{(q,p)}$, for $0<q\leq1$ and $q\leq p<+\infty$. 

\begin{thm}\label{theoremdualunifie}
Suppose that $0<q\leq1$ and $q\leq p<+\infty$. Let $\max\left\{1,p\right\}<r\leq+\infty$. Then, $\left(\mathcal{H}^{(q,p)}\right)^{\ast}$ is isomorphic to $\mathcal{L}_{r',\phi_1,\delta}^{(q,p,\eta)}$ , where $\frac{1}{r}+\frac{1}{r'}=1$, $0<\eta<q$ if $r<+\infty$, and $0<\eta\leq1$ if $r=+\infty$, with equivalent norms. More precisely, we have the following assertions:
\begin{enumerate}
\item Let $g\in\mathcal{L}_{r',\phi_1,\delta}^{(q,p,\eta)}$. Then, the mapping $$T_g:\mathcal{H}_{fin}^{(q,p)}\ni f\longmapsto\int_{\mathbb{R}^d}g(x)f(x)dx,$$ where $\mathcal{H}_{fin}^{(q,p)}$ is the subspace of $\mathcal{H}^{(q,p)}$ consisting of finite linear combinations of $(q,r,\delta)$-atoms, extends to a unique continuous linear functional $\widetilde{T_g}$ on $\mathcal{H}^{(q,p)}$ such that 
\begin{eqnarray*}
\left\|\widetilde{T_g}\right\|=\left\|T_g\right\|\leq C\left\|g\right\|_{\mathcal{L}_{r',\phi_{1},\delta}^{(q,p,\eta)}}, 
\end{eqnarray*} 
where $C>0$ is a constant independent of $g$. 
\item Conversely, for any $T\in\left(\mathcal{H}^{(q,p)}\right)^{\ast}$, there exists $g\in\mathcal{L}_{r',\phi_1,\delta}^{(q,p,\eta)}$ such that $T=T_g$; namely $$T(f)=\int_{\mathbb{R}^d}g(x)f(x)dx,\ \text{ for all }\ f\in\mathcal{H}_{fin}^{(q,p)},$$ and
\begin{eqnarray*}
\left\|g\right\|_{\mathcal{L}_{r',\phi_{1},\delta}^{(q,p,\eta)}}\leq C\left\|T\right\|, 
\end{eqnarray*} 
where $C>0$ is a constant independent of $T$. 
\end{enumerate}
\end{thm}

Notice that in Theorem \ref{theoremdualunifie}, regardless of $\max\left\{1,p\right\}<r\leq+\infty$, we can take $0<\eta\leq1$, when $p\leq1$.\\

Some applications of the results of this section will be given in forthcoming work, in particular the boundedness of Calder\'on-Zygmund operators and convolution operators from $(L^{\infty},\ell^{p'})$ to $\mathcal{L}_{r,\phi_{1},\delta}^{(1,p,\eta)}$, but also from $\mathcal{L}_{r,\phi_{1},\delta}^{(1,p,\eta)}$ to itself, for  $1\leq p<+\infty$, $1\leq r<p'$ and $0<\eta<1$ if $1<r<p'$, or $0<\eta\leq1$ if $r=1$, generalizing classical results between $L^{\infty}$ and $BMO(\mathbb{R}^d)$.

\section{Duality $\mathcal{H}_{\mathrm{loc}}^{(q,p)}-\mathcal{L}_{r,\phi_{1},\delta}^{(q,p,\eta)\mathrm{loc}}$}

As in the preceding section, we give a characterization of the topological dual of $\mathcal{H}_{\mathrm{loc}}^{(q,p)}$ for $0<q\leq1<p<+\infty$. Furthermore, a unified characterization of the dual of $\mathcal{H}_{\mathrm{loc}}^{(q,p)}$, when $0<q\leq1$ and $q\leq p<+\infty$ is obtained.\\ 

Let $1\leq r<+\infty$, $g$ be a function in $L_{\mathrm{loc}}^r$ and $\Omega$ be an open subset such that $\Omega\neq\mathbb{R}^d$. We define $$\textit{O}(g,\Omega,r)^{\mathrm{loc}}:=$$ $$\sup\left[\sum_{|\widetilde{Q^n}|<1}|\widetilde{Q^n}|^{\frac{1}{r'}}\left(\int_{\widetilde{Q^n}}\left|g(x)-P_{\widetilde{Q^n}}^{\delta}(g)(x)\right|^{r}dx\right)^{\frac{1}{r}}+\sum_{|\widetilde{Q^n}|\geq1}|\widetilde{Q^n}|^{\frac{1}{r'}}\left(\int_{\widetilde{Q^n}}|g(x)|^{r}dx\right)^{\frac{1}{r}}\right],$$ where $\frac{1}{r}+\frac{1}{r'}=1$ and the supremum is taken over all families of cubes $\left\{Q^n\right\}_{n\geq0}$ such that $Q^n\subset\Omega$, for all $n\geq0$ and $\sum_{n\geq0}\chi_{_{Q^n}}\leq K(d)$, with $\widetilde{Q^n}=C_0Q^n$, where $K(d)>0$ and $C_0>1$ are the same constants as in the definition of $\textit{O}(g,\Omega,r)$.   

For any open cube $Q$ such that $|Q|<1$, by arguing as in the case of $\textit{O}(g,Q,r)$, we have, by definition of $\textit{O}(g,Q,r)^{\mathrm{loc}}$,  
\begin{eqnarray*}
\textit{O}(g,Q,r)^{\mathrm{loc}}&\geq&|Q|^{\frac{1}{r'}}\left(\int_{Q}\left|g(x)-P_{Q}^{\delta}(g)(x)\right|^{r}dx\right)^{\frac{1}{r}}\\
&=&\left\|\chi_{Q}\right\|_{q,p}\left[\frac{|Q|}{\left\|\chi_{Q}\right\|_{q,p}}\left(\frac{1}{|Q|}\int_{Q}\left|g(x)-P_{Q}^{\delta}(g)(x)\right|^{r}dx\right)^{\frac{1}{r}}\right]. 
\end{eqnarray*} 
Likewise, for any open cube $Q$ such that $|Q|\geq1$, we have   
\begin{eqnarray*}
\textit{O}(g,Q,r)^{\mathrm{loc}}&\geq&|Q|^{\frac{1}{r'}}\left(\int_{Q}|g(x)|^{r}dx\right)^{\frac{1}{r}}\\
&=&\left\|\chi_{Q}\right\|_{q,p}\left[\frac{|Q|}{\left\|\chi_{Q}\right\|_{q,p}}\left(\frac{1}{|Q|}\int_{Q}|g(x)|^{r}dx\right)^{\frac{1}{r}}\right].
\end{eqnarray*} 
Therefore,  
\begin{eqnarray*}
\left\|g\right\|_{\mathcal{L}_{r,\phi_{1},\delta}^{\mathrm{loc}}}&=&\sup_{\underset{|Q|<1}{\underset{Q\text{ open}}{Q\in\mathcal{Q}}}}\frac{1}{\phi_{1}(Q)}\left(\frac{1}{|Q|}\int_{Q}\left|g(x)-P_{Q}^{\delta}(g)(x)\right|^{r}dx\right)^{\frac{1}{r}}\\
&+&\sup_{\underset{|Q|\geq1}{\underset{Q\text{ open}}{Q\in\mathcal{Q}}}}\frac{1}{\phi_{1}(Q)}\left(\frac{1}{|Q|}\int_{Q}|g(x)|^{r}dx\right)^{\frac{1}{r}}\leq\sup_{\underset{|Q|<1}{\underset{Q\text{ open}}{Q\in\mathcal{Q}}}}\frac{\textit{O}(g,Q,r)^{\mathrm{loc}}}{\left\|\chi_{Q}\right\|_{q,p}}\\
&+&\sup_{\underset{|Q|\geq1}{\underset{Q\text{ open}}{Q\in\mathcal{Q}}}}\frac{\textit{O}(g,Q,r)^{\mathrm{loc}}}{\left\|\chi_{Q}\right\|_{q,p}}\leq 2\sup_{\underset{Q\text{ open}}{Q\in\mathcal{Q}}}\frac{\textit{O}(g,Q,r)^{\mathrm{loc}}}{\left\|\chi_{Q}\right\|_{q,p}}\cdot\ \ (\bullet)
\end{eqnarray*}

\begin{defn}
Suppose that $0<q\leq1$ and $0<p<+\infty$. Let $0<\eta<+\infty$ and $1\leq r<+\infty$. We say that a function $g$ in $L_{\mathrm{loc}}^r$ belongs to $\mathcal{L}_{r,\phi_{1},\delta}^{(q,p,\eta)\mathrm{loc}}:=\mathcal{L}_{r,\phi_{1},\delta}^{(q,p,\eta)\mathrm{loc}}(\mathbb{R}^d)$ if there exists a constante $C>0$ such that, for all families of open subsets $\left\{\Omega^j\right\}_{j\in\mathbb{Z}}$ with $\left\|\sum_{j\in\mathbb{Z}}2^{j\eta}\chi_{\Omega^j}\right\|_{\frac{q}{\eta},\frac{p}{\eta}}<+\infty$, we have
\begin{eqnarray}
\sum_{j\in\mathbb{Z}}2^j\textit{O}(g,\Omega^j,r)^{\mathrm{loc}}\leq C\left\|\sum_{j\in\mathbb{Z}}2^{j\eta}\chi_{\Omega^j}\right\|_{\frac{q}{\eta},\frac{p}{\eta}}^{\frac{1}{\eta}}. \label{dualqploc}
\end{eqnarray} 
\end{defn}
 
The set $\mathcal{L}_{r,\phi_{1},\delta}^{(q,p,\eta)\mathrm{loc}}$ is not empty, since the null function is clearly an element of it. We define $\left\|g\right\|_{\mathcal{L}_{r,\phi_{1},\delta}^{(q,p,\eta)\mathrm{loc}}}:=\inf\left\{C>0:\ C \text{ satisfies } (\ref{dualqploc})\right\}$, when $g\in\mathcal{L}_{r,\phi_{1},\delta}^{(q,p,\eta)\mathrm{loc}}$. One can see easily that $\left\|\cdot\right\|_{\mathcal{L}_{r,\phi_{1},\delta}^{(q,p,\eta)\mathrm{loc}}}$ defines a norm on $\mathcal{L}_{r,\phi_{1},\delta}^{(q,p,\eta)\mathrm{loc}}$. From the definition of $\left\|g\right\|_{\mathcal{L}_{r,\phi_{1},\delta}^{(q,p,\eta)\mathrm{loc}}}$ , we have, for all families of open subsets $\left\{\Omega^j\right\}_{j\in\mathbb{Z}}$ of $\mathcal{O}_{\mathbb{Z},q,p,\eta}$,
\begin{eqnarray}
\sum_{j\in\mathbb{Z}}2^j\textit{O}(g,\Omega^j,r)^{\mathrm{loc}}\leq\left\|g\right\|_{\mathcal{L}_{r,\phi_{1},\delta}^{(q,p,\eta)\mathrm{loc}}}\left\|\sum_{j\in\mathbb{Z}}2^{j\eta}\chi_{\Omega^j}\right\|_{\frac{q}{\eta},\frac{p}{\eta}}^{\frac{1}{\eta}}. \label{dualqploc2}
\end{eqnarray}
Thus, if $g\in\mathcal{L}_{r,\phi_{1},\delta}^{(q,p,\eta)\mathrm{loc}}$ , with $0<q\leq1$ and $q\leq p<+\infty$, then 
\begin{eqnarray}
\left\|g\right\|_{\mathcal{L}_{r,\phi_{2},\delta}^{\mathrm{loc}}}\leq\left\|g\right\|_{\mathcal{L}_{r,\phi_{1},\delta}^{\mathrm{loc}}}\leq2\left\|g\right\|_{\mathcal{L}_{r,\phi_{1},\delta}^{(q,p,\eta)\mathrm{loc}}}, \label{dualqploc3}
\end{eqnarray}
and so $\mathcal{L}_{r,\phi_{1},\delta}^{(q,p,\eta)\mathrm{loc}}\hookrightarrow\mathcal{L}_{r,\phi_{1},\delta}^{\mathrm{loc}}\hookrightarrow\mathcal{L}_{r,\phi_{2},\delta}^{\mathrm{loc}}$. In fact,  
\begin{eqnarray*}
\left\|g\right\|_{\mathcal{L}_{r,\phi_{2},\delta}^{\mathrm{loc}}}\leq\left\|g\right\|_{\mathcal{L}_{r,\phi_{1},\delta}^{\mathrm{loc}}}\leq2\sup_{\underset{Q \text{ open}}{Q\in\mathcal{Q}}}\frac{\textit{O}(g,Q,r)^{\mathrm{loc}}}{\left\|\chi_{Q}\right\|_{q,p}}\leq2\left\|g\right\|_{\mathcal{L}_{r,\phi_{1},\delta}^{(q,p,\eta)\mathrm{loc}}},
\end{eqnarray*}
by $(\bullet)$ and the fact that, for any open cube $Q$, we have $\textit{O}(g,Q,r)^{\mathrm{loc}}\leq\left\|g\right\|_{\mathcal{L}_{r,\phi_{1},\delta}^{(q,p,\eta)\mathrm{loc}}}\left\|\chi_{Q}\right\|_{q,p}$. 

As for (\ref{dualqp3}), the inequality on the right in (\ref{dualqploc3}) is valid for $0<q\leq1$ and $0<p<+\infty$. Moreover, when $0<q,p\leq1$ and $0<\eta\leq1$, we have  
\begin{eqnarray}
\left\|g\right\|_{\mathcal{L}_{r,\phi_{1},\delta}^{(q,p,\eta)\mathrm{loc}}}\leq C(d,q,p)\left\|g\right\|_{\mathcal{L}_{r,\phi_{1},\delta}^{\mathrm{loc}}}, \label{dualqploc3bis}
\end{eqnarray}
for any $g\in\mathcal{L}_{r,\phi_{1},\delta}^{\mathrm{loc}}$ , and hence  
\begin{eqnarray}
\left\|g\right\|_{\mathcal{L}_{r,\phi_{1},\delta}^{(q,p,\eta)\mathrm{loc}}}\approx\left\|g\right\|_{\mathcal{L}_{r,\phi_{1},\delta}^{\mathrm{loc}}}, \label{dualqploc3bisbis}
\end{eqnarray}
when $0<q,p\leq1$ and $0<\eta\leq1$, by (\ref{dualqploc3}) and (\ref{dualqploc3bis}). Thus, $\mathcal{L}_{r,\phi_{1},\delta}^{(q,p,\eta)\mathrm{loc}}=\mathcal{L}_{r,\phi_{1},\delta}^{\mathrm{loc}}$ , with equivalent norms, whenever $0<q,p\leq1$ and $0<\eta\leq1$. 
The proof of (\ref{dualqploc3bis}) is similar to the one of (\ref{dualqp3bis}). 

It is also straightforward to see that, for $0<q\leq1$ and $0<p<+\infty$, 
\begin{eqnarray}
\left\|g\right\|_{\mathcal{L}_{r,\phi_{1},\delta}^{(q,p,\eta)}}\leq(1+C)\left\|g\right\|_{\mathcal{L}_{r,\phi_{1},\delta}^{(q,p,\eta)\mathrm{loc}}}, \label{dualqploc3bisbisbis}
\end{eqnarray}
for every $g\in\mathcal{L}_{r,\phi_{1},\delta}^{(q,p,\eta)\mathrm{loc}}$, where the constant $C>0$ is the one of Remark \ref{0laremarquetile0}, and so $\mathcal{L}_{r,\phi_{1},\delta}^{(q,p,\eta)\mathrm{loc}}\hookrightarrow\mathcal{L}_{r,\phi_{1},\delta}^{(q,p,\eta)}$.\\

Our duality result for $\mathcal{H}_{\mathrm{loc}}^{(q,p)}$, with $0<q\leq 1<p<+\infty$, can be stated as follows.\\

\begin{thm} \label{theoremdualqploc}
Suppose that $0<q\leq1<p<+\infty$. Let $p<r\leq+\infty$. Then, $\left(\mathcal{H}_{\mathrm{loc}}^{(q,p)}\right)^{\ast}$ is isomorphic to $\mathcal{L}_{r',\phi_1,\delta}^{(q,p,\eta)\mathrm{loc}}$ , where $\frac{1}{r}+\frac{1}{r'}=1$, $0<\eta<q$ if $r<+\infty$, and $0<\eta\leq1$ if $r=+\infty$, with equivalent norms. More precisely, we have the following assertions:  
\begin{enumerate}
\item Let $g\in\mathcal{L}_{r',\phi_1,\delta}^{(q,p,\eta)\mathrm{loc}}$. Then, the mapping $$T_g:\mathcal{H}_{\mathrm{loc},fin}^{(q,p)}\ni f\longmapsto\int_{\mathbb{R}^d}g(x)f(x)dx,$$ where $\mathcal{H}_{\mathrm{loc},fin}^{(q,p)}$ is the subspace of $\mathcal{H}_{\mathrm{loc}}^{(q,p)}$ consisting of finite linear combinations of $(q,r,\delta)$-atoms, extends to a unique continuous linear functional  
$\widetilde{T_g}$ on $\mathcal{H}_{\mathrm{loc}}^{(q,p)}$ such that  
\begin{eqnarray*}
\left\|\widetilde{T_g}\right\|=\left\|T_g\right\|\leq C\left\|g\right\|_{\mathcal{L}_{r',\phi_{1},\delta}^{(q,p,\eta)\mathrm{loc}}}, 
\end{eqnarray*} 
where $C>0$ is a constant independent of $g$. \label{dualpointqploc1}
\item Conversely, for any $T\in\left(\mathcal{H}_{\mathrm{loc}}^{(q,p)}\right)^{\ast}$, there exists $g\in\mathcal{L}_{r',\phi_1,\delta}^{(q,p,\eta)\mathrm{loc}}$ such that $T=T_g$; namely $$T(f)=\int_{\mathbb{R}^d}g(x)f(x)dx,\ \text{ for all }\ f\in\mathcal{H}_{\mathrm{loc},fin}^{(q,p)},$$ and
\begin{eqnarray*}
\left\|g\right\|_{\mathcal{L}_{r',\phi_{1},\delta}^{(q,p,\eta)\mathrm{loc}}}\leq C\left\|T\right\|, 
\end{eqnarray*} 
where $C>0$ is a constant independent of $T$. \label{dualpointqploc2}
\end{enumerate}
\end{thm}
\begin{proof}
Let $g\in\mathcal{L}_{r',\phi_1,\delta}^{(q,p,\eta)\mathrm{loc}}$. Let $f\in\mathcal{H}_{\mathrm{loc},fin}^{(q,p)}$, where $\mathcal{H}_{\mathrm{loc},fin}^{(q,p)}$ is the subspace of $\mathcal{H}_{\mathrm{loc}}^{(q,p)}$ consisting of finite linear combinations of $(q,r,\delta)$-atoms. Then, $f=\sum_{j=-\infty}^{+\infty}\sum_{k\geq0}\lambda_{j,k}a_{j,k}$ almost everywhere and in the sense of $\mathcal{S'}$, with $\lambda_{j,k}=C_{1}2^j|\widetilde{Q^{j,k}}|^{\frac{1}{q}}$, $\left\|a_{j,k}\right\|_{L^\infty}\leq|\widetilde{Q^{j,k}}|^{-\frac{1}{q}}$, $\left(a_{j,k},\widetilde{Q^{j,k}}\right)\in\mathcal{A}_{\mathrm{loc}}(q,\infty,\delta)$, $\widetilde{Q^{j,k}}=C_0Q^{j,k}$, $Q^{j,k}\subset\mathcal{O}^j$ ($\mathcal{O}^j$ being a certain open subset not equal to $\mathbb{R}^d$), for all $k\geq0$, $\sum_{k\geq0}\chi_{_{Q^{j,k}}}\leq K(d)$ and
\begin{eqnarray*}
\left\|\sum_{j=-\infty}^{+\infty}\sum_{k\geq 0}\left(\frac{|\lambda_{j,k}|}{\left\|\chi_{_{\widetilde{Q^{j,k}}}}\right\|_{q}}\right)^{\eta}\chi_{_{\widetilde{Q^{j,k}}}}\right\|_{\frac{q}{\eta},\frac{p}{\eta}}^{\frac{1}{{\eta}}}\approx \left\|\sum_{j=-\infty}^{+\infty}2^{j\eta}\chi_{\mathcal{O}^j}\right\|_{\frac{q}{\eta},\frac{p}{\eta}}^{\frac{1}{{\eta}}}\approx\left\|f\right\|_{\mathcal{H}_{\mathrm{loc}}^{(q,p)}},
\end{eqnarray*}
$0<\eta\leq1$, by the proof of Theorem 3.9 in \cite{AbFt2} and Theorem 3.2 in \cite{AbFt2}. Notice that each $\left(a_{j,k},\widetilde{Q^{j,k}}\right)$ is also in $\mathcal{A}_{\mathrm{loc}}(q,r,\delta)$ and $f=\sum_{j=-\infty}^{+\infty}\sum_{k\geq0}\lambda_{j,k}a_{j,k}$ (inconditionally) in $\mathcal{H}_{\mathrm{loc}}^{(q,p)}$, but also in $\mathcal{H}_{\mathrm{loc}}^q$ (since $\mathcal{H}_{\mathrm{loc},fin}^{(q,p)}\subset\mathcal{H}_{\mathrm{loc}}^q$). 

For each $j\in\mathbb{Z}$, we put $$G_1^j:=\left\{k\in\mathbb{Z}_{+}: |\widetilde{Q^{j,k}}|<1\right\}$$ and $$G_2^j:=\left\{k\in\mathbb{Z}_{+}: |\widetilde{Q^{j,k}}|\geq1\right\}.$$ Then, since $\mathcal{L}_{r',\phi_{1},\delta}^{(q,p,\eta)\mathrm{loc}}\hookrightarrow\mathcal{L}_{r',\phi_{2},\delta}^{\mathrm{loc}}\hookrightarrow\mathcal{L}_{r',\phi_2,\delta}$ the dual of $\mathcal{H}^q$ (see (\ref{dualqploc3})), we have  
\begin{eqnarray*}
&&\left|\int_{\mathbb{R}^d}g(x)f(x)dx\right|=\left|\sum_{j=-\infty}^{+\infty}\sum_{k\geq0}\lambda_{j,k}\int_{\mathbb{R}^d}g(x)a_{j,k}(x)dx\right|\\
&\leq&\sum_{j=-\infty}^{+\infty}\sum_{k\geq0}|\lambda_{j,k}|\left|\int_{\mathbb{R}^d}g(x)a_{j,k}(x)dx\right|\\
&=&\sum_{j=-\infty}^{+\infty}\sum_{k\in G_1^j}|\lambda_{j,k}|\left|\int_{\mathbb{R}^d}g(x)a_{j,k}(x)dx\right|
+\sum_{j=-\infty}^{+\infty}\sum_{k\in G_2^j}|\lambda_{j,k}|\left|\int_{\mathbb{R}^d}g(x)a_{j,k}(x)dx\right|\\
&=&:\ell_1+\ell_2,
\end{eqnarray*}
with, as in the proof of Theorem \ref{theoremdualqp}, 
\begin{eqnarray*}
\ell_1&=&\sum_{j=-\infty}^{+\infty}\sum_{k\in G_1^j}C_{1}2^j|\widetilde{Q^{j,k}}|^{\frac{1}{q}}\left|\int_{\widetilde{Q^{j,k}}}(g(x)-P_{\widetilde{Q^{j,k}}}^\delta(g)(x))a_{j,k}(x)dx\right|\\
&\leq&C_{1}\sum_{j=-\infty}^{+\infty}2^j\sum_{k\in G_1^j}|\widetilde{Q^{j,k}}|^{\frac{1}{r}}\left(\int_{\widetilde{Q^{j,k}}}|g(x)-P_{\widetilde{Q^{j,k}}}^\delta(g)(x)|^{r'}dx\right)^{\frac{1}{r'}}\\
&\leq&C_{1}\sum_{j=-\infty}^{+\infty}2^j\textit{O}(g,\mathcal{O}^j,r')^{\mathrm{loc}}\\
&\leq& C_{1}\left\|g\right\|_{\mathcal{L}_{r',\phi_{1},\delta}^{(q,p,\eta)\mathrm{loc}}}\left\|\sum_{j=-\infty}^{+\infty}2^{j\eta}\chi_{\mathcal{O}^j}\right\|_{\frac{q}{\eta},\frac{p}{\eta}}^{\frac{1}{\eta}}\leq C\left\|g\right\|_{\mathcal{L}_{r',\phi_{1},\delta}^{(q,p,\eta)\mathrm{loc}}}\left\|f\right\|_{\mathcal{H}_{\mathrm{loc}}^{(q,p)}}
\end{eqnarray*}
and  
\begin{eqnarray*}
\ell_2&=&\sum_{j=-\infty}^{+\infty}\sum_{k\in G_2^j}C_{1}2^j|\widetilde{Q^{j,k}}|^{\frac{1}{q}}\left|\int_{\widetilde{Q^{j,k}}}g(x)a_{j,k}(x)dx\right|\\
&\leq&C_{1}\sum_{j=-\infty}^{+\infty}2^j\sum_{k\in G_2^j}|\widetilde{Q^{j,k}}|^{\frac{1}{r}}\left(\int_{\widetilde{Q^{j,k}}}|g(x)|^{r'}dx\right)^{\frac{1}{r'}}\leq C_{1}\sum_{j=-\infty}^{+\infty}2^j\textit{O}(g,\mathcal{O}^j,r')^{\mathrm{loc}}\\
&\leq&C_{1}\left\|g\right\|_{\mathcal{L}_{r',\phi_{1},\delta}^{(q,p,\eta)\mathrm{loc}}}\left\|\sum_{j=-\infty}^{+\infty}2^{j\eta}\chi_{\mathcal{O}^j}\right\|_{\frac{q}{\eta},\frac{p}{\eta}}^{\frac{1}{\eta}}\leq C\left\|g\right\|_{\mathcal{L}_{r',\phi_{1},\delta}^{(q,p,\eta)\mathrm{loc}}}\left\|f\right\|_{\mathcal{H}_{\mathrm{loc}}^{(q,p)}}.
\end{eqnarray*}
Hence   
\begin{eqnarray*}
|T_g(f)|=\left|\int_{\mathbb{R}^d}g(x)f(x)dx\right|\leq C\left\|g\right\|_{\mathcal{L}_{r',\phi_{1},\delta}^{(q,p,\eta)\mathrm{loc}}}\left\|f\right\|_{\mathcal{H}_{\mathrm{loc}}^{(q,p)}}.
\end{eqnarray*} 
Therefore, $T_g$ extends to a unique continuous linear functional $\widetilde{T_g}$ on $\mathcal{H}_{\mathrm{loc}}^{(q,p)}$, with $\left\|\widetilde{T_g}\right\|=\left\|T_g\right\|\leq C\left\|g\right\|_{\mathcal{L}_{r',\phi_{1},\delta}^{(q,p,\eta)\mathrm{loc}}}$ and $C>0$ is a constant independent of $g$.\\ 

Conversely, let $T\in\left(\mathcal{H}_{\mathrm{loc}}^{(q,p)}\right)^{\ast}$. Then, there exists $g\in\mathcal{L}_{r',\phi_1,\delta}^{\mathrm{loc}}$ such that $T=T_g$; namely $$T(f)=T_g(f)=\int_{\mathbb{R}^d}g(x)f(x)dx,\ \text{ for all }\ f\in\mathcal{H}_{\mathrm{loc},fin}^{(q,p)},\ \ (\oslash)$$ by the proof of Theorem \ref{theoremdualloc} (\cite{AbFt2}, Theorem 4.3). We are going to show that $g$ belongs to $\in\mathcal{L}_{r',\phi_1,\delta}^{(q,p,\eta)\mathrm{loc}}$. Let $\left\{\Omega^j\right\}_{j\in\mathbb{Z}}$ be a family of open subsets such that $\left\|\sum_{j\in\mathbb{Z}}2^{j\eta}\chi_{\Omega^j}\right\|_{\frac{q}{\eta},\frac{p}{\eta}}<+\infty$, and $\left\{Q^{j,k}\right\}_{k\geq0}$ be a family of cubes such that $Q^{j,k}\subset\Omega^j$, for all $k\geq0$, $\sum_{k\geq0}\chi_{_{Q^{j,k}}}\leq K(d)$, with $\widetilde{Q^{j,k}}:=C_0Q^{j,k}$, and 
\begin{eqnarray*}
\textit{O}(g,\Omega^j,r')^{\mathrm{loc}}&\leq&2\sum_{k\geq0:\ |\widetilde{Q^{j,k}}|<1}|\widetilde{Q^{j,k}}|^{\frac{1}{r}}\left(\int_{\widetilde{Q^{j,k}}}|g(x)-P_{\widetilde{Q^{j,k}}}^\delta(g)(x)|^{r'}dx\right)^{\frac{1}{r'}}\\
&+&2\sum_{k\geq0:\ |\widetilde{Q^{j,k}}|\geq1}|\widetilde{Q^{j,k}}|^{\frac{1}{r}}\left(\int_{\widetilde{Q^{j,k}}}|g(x)|^{r'}dx\right)^{\frac{1}{r'}}.\ \ (\oslash\oslash) 
\end{eqnarray*}
Here, the hypothesis $(\oslash\oslash)$ is also justified and its justification is similar to the one of $(\otimes\otimes)$. 

Let $\widetilde{Q^{j,k}}$ with $|\widetilde{Q^{j,k}}|\geq1$ and, $h_{j,k}\in L^{r}(\widetilde{Q^{j,k}})$ such that $\left\|h_{j,k}\right\|_{L^r(\widetilde{Q^{j,k}})}=1$ and $$\left(\int_{\widetilde{Q^{j,k}}}|g(x)|^{r'}dx\right)^{\frac{1}{r'}}=\int_{\widetilde{Q^{j,k}}}g(x)h_{j,k}(x)dx.$$ Put $$a_{j,k}(x):=|\widetilde{Q^{j,k}}|^{\frac{1}{r}-\frac{1}{q}}h_{j,k}(x).$$ We have $(a_{j,k},\widetilde{Q^{j,k}})\in\mathcal{A}_{\mathrm{loc}}(q,r,\delta)$. Put $\lambda_{j,k}:=2^j|\widetilde{Q^{j,k}}|^{\frac{1}{q}}$. As in the proof of Theorem \ref{theoremdualqp}, we have 
\begin{eqnarray*}
&&\left\|\sum_{j\in\mathbb{Z}}\sum_{k\geq0:\ |\widetilde{Q^{j,k}}|\geq1}\left(\frac{|\lambda_{j,k}|}{\left\|\chi_{_{\widetilde{Q^{j,k}}}}\right\|_{q}}\right)^{\eta}\chi_{_{\widetilde{Q^{j,k}}}}\right\|_{\frac{q}{\eta},\frac{p}{\eta}}^{\frac{1}{{\eta}}}\leq C\left\|\sum_{j\in\mathbb{Z}}\sum_{k\geq0:\ |\widetilde{Q^{j,k}}|\geq1}2^{j\eta}\chi_{_{Q^{j,k}}}\right\|_{\frac{q}{\eta},\frac{p}{\eta}}^{\frac{1}{\eta}}\\
&\leq&C\left\|\sum_{j\in\mathbb{Z}}\sum_{k\geq0}2^{j\eta}\chi_{_{Q^{j,k}}}\right\|_{\frac{q}{\eta},\frac{p}{\eta}}^{\frac{1}{\eta}}\leq C\left\|\sum_{j\in\mathbb{Z}}2^{j\eta}\chi_{\Omega^j}\right\|_{\frac{q}{\eta},\frac{p}{\eta}}^{\frac{1}{\eta}}<+\infty.
\end{eqnarray*} 
Then, considering that $0<\eta<q$ (if $p<r<+\infty$) or $0<\eta\leq 1$ (if $r=+\infty$), we have $\sum_{j\in\mathbb{Z}}\sum_{k\geq0:\ |\widetilde{Q^{j,k}}|\geq1}\lambda_{j,k}a_{j,k}\in\mathcal{H}_{\mathrm{loc}}^{(q,p)}$ and
\begin{eqnarray*}
\left\|\sum_{j\in\mathbb{Z}}\sum_{k\geq0:\ |\widetilde{Q^{j,k}}|\geq1}\lambda_{j,k}a_{j,k}\right\|_{\mathcal{H}_{\mathrm{loc}}^{(q,p)}}&\leq&C\left\|\sum_{j\in\mathbb{Z}}\sum_{k\geq0:\ |\widetilde{Q^{j,k}}|\geq1}\left(\frac{|\lambda_{j,k}|}{\left\|\chi_{_{\widetilde{Q^{j,k}}}}\right\|_{q}}\right)^{\eta}\chi_{_{\widetilde{Q^{j,k}}}}\right\|_{\frac{q}{\eta},\frac{p}{\eta}}^{\frac{1}{{\eta}}}\\
&\leq&C\left\|\sum_{j\in\mathbb{Z}}2^{j\eta}\chi_{\Omega^j}\right\|_{\frac{q}{\eta},\frac{p}{\eta}}^{\frac{1}{\eta}},\ \ (\oslash\oslash\oslash)
\end{eqnarray*}
by Theorem 3.3 in \cite{AbFt2} (if $p<r<+\infty$) or Theorem 3.2 in \cite{AbFt2} (if $r=+\infty$). 

Now, let $\widetilde{Q^{j,k}}$ with $|\widetilde{Q^{j,k}}|<1$ and $h_{j,k}\in L^{r}(\widetilde{Q^{j,k}})$ such that $\left\|h_{j,k}\right\|_{L^r(\widetilde{Q^{j,k}})}=1$ and $$\left(\int_{\widetilde{Q^{j,k}}}\left|g(x)-P_{\widetilde{Q^{j,k}}}^{\delta}(g)(x)\right|^{r'}dx\right)^{\frac{1}{r'}}=\int_{\widetilde{Q^{j,k}}}\left(g(x)-P_{\widetilde{Q^{j,k}}}^{\delta}(g)(x)\right)h_{j,k}(x)dx.$$ By putting $$a_{j,k}(x):=C_{h_{j,k},Q^{j,k},\delta}|\widetilde{Q^{j,k}}|^{\frac{1}{r}-\frac{1}{q}}\left(h_{j,k}(x)-P_{\widetilde{Q^{j,k}}}^{\delta}(h_{j,k})(x)\right)\chi_{\widetilde{Q^{j,k}}}(x),$$ with $C_{h_{j,k},Q^{j,k},\delta}:=\left(1+\left\|\left(h_{j,k}-P_{\widetilde{Q^{j,k}}}^{\delta}(h_{j,k})\right)\chi_{\widetilde{Q^{j,k}}}\right\|_r\right)^{-1},$ and $\lambda_{j,k}:=2^j|\widetilde{Q^{j,k}}|^{\frac{1}{q}}$, we have, as in the proof of Theorem \ref{theoremdualqp}, $(a_{j,k},\widetilde{Q^{j,k}})\in\mathcal{A}_{\mathrm{loc}}(q,r,\delta)$ and 
\begin{eqnarray*}
\left\|\sum_{j\in\mathbb{Z}}\sum_{k\geq 0:\ |\widetilde{Q^{j,k}}|<1}\left(\frac{|\lambda_{j,k}|}{\left\|\chi_{_{\widetilde{Q^{j,k}}}}\right\|_{q}}\right)^{\eta}\chi_{_{\widetilde{Q^{j,k}}}}\right\|_{\frac{q}{\eta},\frac{p}{\eta}}^{\frac{1}{{\eta}}}\leq C\left\|\sum_{j\in\mathbb{Z}}2^{j\eta}\chi_{\Omega^j}\right\|_{\frac{q}{\eta},\frac{p}{\eta}}^{\frac{1}{\eta}}<+\infty. 
\end{eqnarray*} 
Hence considering that $0<\eta<q$ (if $p<r<+\infty$) or $0<\eta\leq 1$  (if $r=+\infty$), we have $\sum_{j\in\mathbb{Z}}\sum_{k\geq 0:\ |\widetilde{Q^{j,k}}|<1}\lambda_{j,k}a_{j,k}\in\mathcal{H}_{\mathrm{loc}}^{(q,p)}$ and
\begin{eqnarray*}
\left\|\sum_{j\in\mathbb{Z}}\sum_{k\geq 0:\ |\widetilde{Q^{j,k}}|<1}\lambda_{j,k}a_{j,k}\right\|_{\mathcal{H}_{\mathrm{loc}}^{(q,p)}}&\leq&C
\left\|\sum_{j\in\mathbb{Z}}\sum_{k\geq 0:\ |\widetilde{Q^{j,k}}|<1}\left(\frac{|\lambda_{j,k}|}{\left\|\chi_{_{\widetilde{Q^{j,k}}}}\right\|_{q}}\right)^{\eta}\chi_{_{\widetilde{Q^{j,k}}}}\right\|_{\frac{q}{\eta},\frac{p}{\eta}}^{\frac{1}{{\eta}}}\\
&\leq&C\left\|\sum_{j\in\mathbb{Z}}2^{j\eta}\chi_{\Omega^j}\right\|_{\frac{q}{\eta},\frac{p}{\eta}}^{\frac{1}{\eta}},\ \ (\oslash\oslash\oslash\ \oslash)
\end{eqnarray*}
by Theorem 3.3 in \cite{AbFt2} (if $p<r<+\infty$) or Theorem 3.2 in \cite{AbFt2} (if $r=+\infty$). 

Furthermore, 
\begin{eqnarray*}
\sum_{j\in\mathbb{Z}}2^j\textit{O}(g,\Omega^j,r')^{\mathrm{loc}}&\leq&2\sum_{j\in\mathbb{Z}}2^j\sum_{k\geq0:\ |\widetilde{Q^{j,k}}|<1}|\widetilde{Q^{j,k}}|^{\frac{1}{r}}\left(\int_{\widetilde{Q^{j,k}}}|g(x)-P_{\widetilde{Q^{j,k}}}^\delta(g)(x)|^{r'}dx\right)^{\frac{1}{r'}}\\
&+&2\sum_{j\in\mathbb{Z}}2^j\sum_{k\geq0:\ |\widetilde{Q^{j,k}}|\geq1}|\widetilde{Q^{j,k}}|^{\frac{1}{r}}\left(\int_{\widetilde{Q^{j,k}}}|g(x)|^{r'}dx\right)^{\frac{1}{r'}}\\
&=&2\sum_{j\in\mathbb{Z}}2^j\sum_{k\geq0:\ |\widetilde{Q^{j,k}}|<1}|\widetilde{Q^{j,k}}|^{\frac{1}{r}}\int_{\widetilde{Q^{j,k}}}(g(x)-P_{\widetilde{Q^{j,k}}}^{\delta}(g)(x))h_{j,k}(x)dx\\
&+&2\sum_{j\in\mathbb{Z}}2^j\sum_{k\geq0:\ |\widetilde{Q^{j,k}}|\geq1}|\widetilde{Q^{j,k}}|^{\frac{1}{r}}\int_{\widetilde{Q^{j,k}}}g(x)h_{j,k}(x)dx,
\end{eqnarray*}
by $(\oslash\oslash)$. But, proceeding again as in the proof of Theorem \ref{theoremdualqp}, we have   
\begin{eqnarray*}
&&\sum_{j\in\mathbb{Z}}2^j\sum_{k\geq0:\ |\widetilde{Q^{j,k}}|\geq1}|\widetilde{Q^{j,k}}|^{\frac{1}{r}}\int_{\widetilde{Q^{j,k}}}g(x)h_{j,k}(x)dx\\&=&\sum_{j\in\mathbb{Z}}\sum_{k\geq0:\ |\widetilde{Q^{j,k}}|\geq1}2^j|\widetilde{Q^{j,k}}|^{\frac{1}{q}}\int_{\widetilde{Q^{j,k}}}g(x)a_{j,k}(x)dx=\sum_{j\in\mathbb{Z}}\sum_{k\geq0:\ |\widetilde{Q^{j,k}}|\geq1}\lambda_{j,k}T(a_{j,k})\\
&=&T\left(\sum_{j\in\mathbb{Z}}\sum_{k\geq0:\ |\widetilde{Q^{j,k}}|\geq1}\lambda_{j,k}a_{j,k}\right)\leq C\left\|T\right\|\left\|\sum_{j\in\mathbb{Z}}2^{j\eta}\chi_{\Omega^j}\right\|_{\frac{q}{\eta},\frac{p}{\eta}}^{\frac{1}{\eta}},
\end{eqnarray*}
by $(\oslash)$ and $(\oslash\oslash\oslash)$, and
\begin{eqnarray*}
&&\sum_{j\in\mathbb{Z}}2^j\sum_{k\geq0:\ |\widetilde{Q^{j,k}}|<1}|\widetilde{Q^{j,k}}|^{\frac{1}{r}}\int_{\widetilde{Q^{j,k}}}(g(x)-P_{\widetilde{Q^{j,k}}}^{\delta}(g)(x))h_{j,k}(x)dx\\
&\leq&C\sum_{j\in\mathbb{Z}}\sum_{k\geq0:\ |\widetilde{Q^{j,k}}|<1}2^j|\widetilde{Q^{j,k}}|^{\frac{1}{q}}\int_{\widetilde{Q^{j,k}}}g(x)a_{j,k}(x)dx=C\sum_{j\in\mathbb{Z}}\sum_{k\geq0:\ |\widetilde{Q^{j,k}}|<1}\lambda_{j,k}T(a_{j,k})\\
&=&CT\left(\sum_{j\in\mathbb{Z}}\sum_{k\geq0:\ |\widetilde{Q^{j,k}}|<1}\lambda_{j,k}a_{j,k}\right)\leq C\left\|T\right\|\left\|\sum_{j\in\mathbb{Z}}2^{j\eta}\chi_{\Omega^j}\right\|_{\frac{q}{\eta},\frac{p}{\eta}}^{\frac{1}{\eta}},
\end{eqnarray*}
by $(\oslash)$ and $(\oslash\oslash\oslash\ \oslash)$. Therefore, 
\begin{eqnarray*}
\sum_{j\in\mathbb{Z}}2^j\textit{O}(g,\Omega^j,r')^{\mathrm{loc}}&\leq& CT\left(\sum_{j\in\mathbb{Z}}\sum_{k\geq0:\ |\widetilde{Q^{j,k}}|\geq1}\lambda_{j,k}a_{j,k}\right)\\
&+&CT\left(\sum_{j\in\mathbb{Z}}\sum_{k\geq0:\ |\widetilde{Q^{j,k}}|<1}\lambda_{j,k}a_{j,k}\right)\leq C\left\|T\right\|\left\|\sum_{j\in\mathbb{Z}}2^{j\eta}\chi_{\Omega^j}\right\|_{\frac{q}{\eta},\frac{p}{\eta}}^{\frac{1}{\eta}}.
\end{eqnarray*} 
Hence $g\in\mathcal{L}_{r',\phi_1,\delta}^{(q,p,\eta)\mathrm{loc}}$ and $\left\|g\right\|_{\mathcal{L}_{r',\phi_1,\delta}^{(q,p,\eta)\mathrm{loc}}}\leq C\left\|T\right\|$. This ends the proof of Theorem \ref{theoremdualqploc}.\\
\end{proof}

We have also the local version of Remark \ref{remarqedualeqp1}; namely: 

\begin{remark}\label{remarqedualeqploc1}
Suppose that $0<q\leq1<p<+\infty$. Let $1<r<p'$ , $0<\eta<q$ and $0<\eta_1\leq1$. We have $$\mathcal{L}_{1,\phi_1,\delta}^{(q,p,\eta_1)\mathrm{loc}}\cong\left(\mathcal{H}_{\mathrm{loc}}^{(q,p)}\right)^{\ast}\cong\mathcal{L}_{r,\phi_1,\delta}^{(q,p,\eta)\mathrm{loc}}\hookrightarrow\mathcal{L}_{r,\phi_2,\delta}^{\mathrm{loc}}\cong\left(\mathcal{H}_{\mathrm{loc}}^q\right)^{\ast}$$ and $$\mathcal{L}_{1,\phi_1,\delta}^{(q,p,\eta_1)\mathrm{loc}}\cong\left(\mathcal{H}_{\mathrm{loc}}^{(q,p)}\right)^{\ast}\cong\mathcal{L}_{r,\phi_1,\delta}^{(q,p,\eta)\mathrm{loc}}\hookrightarrow L^{p'}\cong\left(\mathcal{H}_{\mathrm{loc}}^p\right)^{\ast},$$ by Theorem \ref{theoremdualqploc}, because $L^p\cong\mathcal{H}_{\mathrm{loc}}^p\hookrightarrow\mathcal{H}_{\mathrm{loc}}^{(q,p)}$. Thus, when $q=1$, we have $$(L^{\infty},\ell^{p'})\hookrightarrow\mathcal{L}_{1,\phi_1,\delta}^{(1,p,\eta_1)\mathrm{loc}}\cong\mathcal{L}_{r,\phi_1,\delta}^{(1,p,\eta)\mathrm{loc}}\hookrightarrow bmo(\mathbb{R}^d)$$ and $$(L^{\infty},\ell^{p'})\hookrightarrow\mathcal{L}_{1,\phi_1,\delta}^{(1,p,\eta_1)\mathrm{loc}}\cong\mathcal{L}_{r,\phi_1,\delta}^{(1,p,\eta)\mathrm{loc}}\hookrightarrow L^{p'},$$ since $\mathcal{H}_{\mathrm{loc}}^1\hookrightarrow\mathcal{H}_{\mathrm{loc}}^{(1,p)}\hookrightarrow(L^1,\ell^p)$ (see Theorem 3.2 in \cite{AbFt}), $(L^1,\ell^p)^{\ast}\cong(L^{\infty},\ell^{p'})$ and $\mathcal{L}_{r,\phi_2,\delta}^{\mathrm{loc}}\cong(\mathcal{H}_{\mathrm{loc}}^1)^{\ast}\cong bmo(\mathbb{R}^d)$.
\end{remark}  

Here, by combining also Theorem \ref{theoremdualloc} with (\ref{dualqploc3bisbis}), and Theorem \ref{theoremdualqploc}, we obtain a unified characterization of the topological dual of Hardy-amalgam spaces $\mathcal{H}_{\mathrm{loc}}^{(q,p)}$, for $0<q\leq1$ and $q\leq p<+\infty$.

\begin{thm} \label{theoremdualunifieloc}
Suppose that  $0<q\leq1$ and $q\leq p<+\infty$. Let $\max\left\{1,p\right\}<r\leq+\infty$. Then, $\left(\mathcal{H}_{\mathrm{loc}}^{(q,p)}\right)^{\ast}$ is isomorphic to $\mathcal{L}_{r',\phi_1,\delta}^{(q,p,\eta)\mathrm{loc}}$ , where $\frac{1}{r}+\frac{1}{r'}=1$, $0<\eta<q$ if $r<+\infty$, and $0<\eta\leq1$ if $r=+\infty$, with equivalent norms. More precisely, we have the following assertions:  
\begin{enumerate}
\item Let $g\in\mathcal{L}_{r',\phi_1,\delta}^{(q,p,\eta)\mathrm{loc}}$. Then, the mapping $$T_g:\mathcal{H}_{\mathrm{loc},fin}^{(q,p)}\ni f\longmapsto\int_{\mathbb{R}^d}g(x)f(x)dx,$$ where $\mathcal{H}_{\mathrm{loc},fin}^{(q,p)}$ is the subspace of $\mathcal{H}_{\mathrm{loc}}^{(q,p)}$ consisting of finite linear combinations of $(q,r,\delta)$-atoms, extends to a unique continuous linear functional  
$\widetilde{T_g}$ on $\mathcal{H}_{\mathrm{loc}}^{(q,p)}$ such that  
\begin{eqnarray*}
\left\|\widetilde{T_g}\right\|=\left\|T_g\right\|\leq C\left\|g\right\|_{\mathcal{L}_{r',\phi_{1},\delta}^{(q,p,\eta)\mathrm{loc}}}, 
\end{eqnarray*} 
where $C>0$ is a constant independent of $g$. 
\item Conversely, for any $T\in\left(\mathcal{H}_{\mathrm{loc}}^{(q,p)}\right)^{\ast}$, there exists $g\in\mathcal{L}_{r',\phi_1,\delta}^{(q,p,\eta)\mathrm{loc}}$ such that $T=T_g$; namely $$T(f)=\int_{\mathbb{R}^d}g(x)f(x)dx,\ \text{ for all }\ f\in\mathcal{H}_{\mathrm{loc},fin}^{(q,p)},$$ and
\begin{eqnarray*}
\left\|g\right\|_{\mathcal{L}_{r',\phi_{1},\delta}^{(q,p,\eta)\mathrm{loc}}}\leq C\left\|T\right\|, 
\end{eqnarray*} 
where $C>0$ is a constant independent of $T$. 
\end{enumerate}
\end{thm}

In Theorem \ref{theoremdualunifieloc}, regardless of $\max\left\{1,p\right\}<r\leq+\infty$, we can take $0<\eta\leq1$, when $p\leq1$.\\

Thanks to Theorems  \ref{theoremdualunifie} and \ref{theoremdualunifieloc}, for $0<q\leq1$, $q\leq p<+\infty$ and, $1\leq r<p'$ if $1<p$ and $1\leq r<+\infty$ otherwise, where $\frac{1}{p}+\frac{1}{p'}=1$, with $0<\eta<q$ if $1<r$, and $0<\eta\leq1$ if $r=1$, we have 
\begin{eqnarray}
\left(\mathcal{H}_{\mathrm{loc}}^{(q,p)}\right)^{\ast}\cong\mathcal{L}_{r,\phi_1,\delta}^{(q,p,\eta)\mathrm{loc}}\hookrightarrow\mathcal{L}_{r,\phi_1,\delta}^{(q,p,\eta)}\cong\left(\mathcal{H}^{(q,p)}\right)^{\ast}. \label{inclusqp}
\end{eqnarray}
Moreover, the inclusion in (\ref{inclusqp}) is strict; namely $\mathcal{L}_{r,\phi_1,\delta}^{(q,p,\eta)\mathrm{loc}}\subsetneq\mathcal{L}_{r,\phi_1,\delta}^{(q,p,\eta)}$. 
 
Indeed, we know that $\mathcal{L}_{r,\phi_1,\delta}^{(q,p,\eta)\mathrm{loc}}\subset\mathcal{L}_{r,\phi_1,\delta}^{(q,p,\eta)}$ , by (\ref{dualqploc3bisbisbis}), and $\mathcal{P_{\delta}}\subset\mathcal{L}_{r,\phi_1,\delta}^{(q,p,\eta)}$. In other respects,  
\begin{eqnarray}
\mathcal{P_{\delta}}\not\subset\mathcal{L}_{r,\phi_1,\delta}^{(q,p,\eta)\mathrm{loc}}. \label{compardeshqp1}
\end{eqnarray}
We infer from these relations the result. Hence, we must prove (\ref{compardeshqp1}) to finish. Relation (\ref{compardeshqp1}) is clear when $0<q<1$, since $\mathcal{P_{\delta}}\not\subset L^{\infty}$ and $\mathcal{L}_{r,\phi_1,\delta}^{(q,p,\eta)\mathrm{loc}}\subset\mathcal{L}_{r,\phi_2,\delta}^{\mathrm{loc}}=\Lambda_{d\left(\frac{1}{q}-1\right)}\subset L^{\infty}$, where $\Lambda_{d\left(\frac{1}{q}-1\right)}$ is the dual of $\mathcal{H}_{\mathrm{loc}}^q$ defined by D. Goldberg \cite{DGG}. When $q=1$, or more generally $0<q\leq1<p<+\infty$, Relation (\ref{compardeshqp1}) is proved as follows. 
 
We know that $\mathcal{L}_{r,\phi_1,\delta}^{(q,p,\eta)\mathrm{loc}}\subset\mathcal{L}_{r,\phi_1,\delta}^{\mathrm{loc}}$ , by (\ref{dualqploc3}). Let $0\neq c\in\mathbb{C}$. Consider the constant function $g=c$. Then, $g\in\mathcal{P_{\delta}}$, and so

\begin{eqnarray*}
\left\|g\right\|_{\mathcal{L}_{r,\phi_1,\delta}^{\mathrm{loc}}}&=&\sup_{\underset{|Q|\geq1}{Q\in\mathcal{Q}}}\frac{|Q|}{\left\|\chi_Q\right\|_{q,p}}\left(\frac{1}{|Q|}\int_{Q}|g(x)|^rdx\right)^{\frac{1}{r}}\\
&\geq&\sup_{\underset{|Q|\geq1}{Q\in\mathcal{Q}}}\frac{|Q|}{\left\|\chi_Q\right\|_p}\left(\frac{1}{|Q|}\int_{Q}|g(x)|^rdx\right)^{\frac{1}{r}}=|c|\sup_{\underset{|Q|\geq1}{Q\in\mathcal{Q}}}|Q|^{1-\frac{1}{p}}=+\infty,
\end{eqnarray*}
since $1-\frac{1}{p}>0$ implies that $|Q|^{1-\frac{1}{p}}\uparrow+\infty$ as $|Q|\uparrow+\infty$. Hence $g\notin\mathcal{L}_{r,\phi_1,\delta}^{\mathrm{loc}}$. We deduce from it that $\mathcal{P_{\delta}}\not\subset\mathcal{L}_{r,\phi_1,\delta}^{(q,p,\eta)\mathrm{loc}}$, since $\mathcal{L}_{r,\phi_1,\delta}^{(q,p,\eta)\mathrm{loc}}\subset\mathcal{L}_{r,\phi_1,\delta}^{\mathrm{loc}}$; this states (\ref{compardeshqp1}).

Now we know that the dual of $\mathcal{H}_{\mathrm{loc}}^{(q,p)}$ is strictly embedded in the dual of $\mathcal{H}^{(q,p)}$, then is it possible to claim that $\mathcal{H}^{(q,p)}$ is strictly embedded in $\mathcal{H}_{\mathrm{loc}}^{(q,p)}$, for $0<q\leq1<p<+\infty$ in particular? This question will be studied in our forthcoming paper.\\

Another consequence of Theorem \ref{theoremdualunifieloc} is the following result that extends Corallary 5.6 in \cite{AbFt2}.
 
\begin{cor} 
Suppose that $0<q\leq1$ and $q\leq p<+\infty$. Let $\max\left\{p,1\right\}<r\leq+\infty$ and $T$ be a $\mathcal{S}^0$ pseudo-differential operator. Then, there exists a positive constant $C>0$ such that, for all $f$ in $\mathcal{L}_{r',\phi_1,\delta}^{(q,p,\eta)\mathrm{loc}}$, where $0<\eta<q$, if $r<+\infty$, and $0<\eta\leq1$ if $r=+\infty$, 
\begin{eqnarray*}
\left\|T(f)\right\|_{\mathcal{L}_{r',\phi_{1},\delta}^{(q,p,\eta)\mathrm{loc}}}\leq C\left\|f\right\|_{\mathcal{L}_{r',\phi_{1},\delta}^{(q,p,\eta)\mathrm{loc}}}.
\end{eqnarray*}
\end{cor}


\begin{thebibliography}{MTW1}
\bibitem{AbFt} Z. V. d. P. Abl\'e and J. Feuto, \textit{Atomic decomposition of Hardy-amalgam spaces}, J. Math. Anal. Appl. 455 (2017), 1899-1936. 
\bibitem{AbFt1} Z. V. P. Abl\'e and J. Feuto, \textit{Dual of Hardy-amalgam spaces and norms inequalities}, Analysis Math., 45 (4) (2019), 647-686. 
\bibitem{AbFt2} Z. V. P. Abl\'e and J. Feuto, \textit{Dual of Hardy-amalgam spaces $\mathcal{H}_{\mathrm{loc}}^{(q,p)}$ and pseudo-differential operators}, Preprint 2018 arXiv: 1803.03595.
\bibitem{AT} W. Abu-Shammala, \textit{The Hardy-Lorentz spaces}, Thesis of Indiana University, (2007).
\bibitem{BDD} J. P. Bertrandias, C. Datry and C. Dupuis, \textit{Unions et intersections d’espaces} $L^p$ \textit{invariantes par translation ou convolution}, Ann. Inst. Fourier (Grenoble), 28 (1978), 53-84.
\bibitem{MBOW} M. Bownik, \textit{Anisotropic Hardy spaces and waveletes}, Mem. Amer. Math. Soc, 164 (2003).
\bibitem{RBHS} R. C. Busby and H. A. Smith, \textit{Product-convolution operators and mixed-norm spaces}, Trans. Amer. Math. Soc., 263 (1981), 309-341.
\bibitem{CLHS} C. Carton-Lebrun, H. P. Heinig and S.C. Hofmann, \textit{Integral operators on weighted amalgams}, Studia Math., 109 (1994), 133-157.
\bibitem{FSTW} J. J. F. Fournier and J. Stewart, \textit{Amalgams of} $L^p$ and $l^p$, Bull. Amer. Math. Soc., 13 (1) (1985), 1-21. 
\bibitem{JGLR} J. Garc\'ia-Cuerva, J. L. Rubio de Francia, \textit{Weighted Norm Inequalities and Related Topics}, North-Holland, 1985.
\bibitem{DGG} D. Goldberg, \textit{A local version of real Hardy spaces}, Duke Math. Journal., 46 (1979), 27-42. 
\bibitem{FH} F. Holland, \textit{Harmonic analysis on amalgams of} $L^p$ and $l^q$, J. London Math. Soc., 2 (10) (1975), 295-305.
\bibitem{LSUYY} Y. Liang, Y. Sawano, T. Ullrich, D. Yang and W. Yuan, \textit{ A new framework for generalized 
Besov-type and Triebel-Lizorkin-type spaces}, Dissertationes Math. (Rozprawy Mat.). 489 (2013), 1-114. 
\bibitem{SZLU} S. Z. Lu, \textit{Four Lectures on Real $H^p$ Spaces}, World Scientific Publishing Co., Inc., River Edge, NJ, 1995.
\bibitem{NEYS} E. Nakai and Y. Sawano, \textit{Hardy spaces with variable exponents and generalized Campanato
spaces}, J. Funct. Anal., 262, (2012), 3665-3748.
\bibitem{MA} E. M. Stein, \textit{Harmonic analysis: real-variable methods, orthogonality, and oscillatory integral}, Princeton Mathematical Series, 43. Monographs in Harmonic Analysis, III.Princeton University Press, Princeton, NJ, 1993. 
\bibitem{JSTW} J. Stewart, \textit{Fourier transforms of unbounded measures}, Canad. J. Math., 31 (1979), 1281-1292.
\bibitem{YDYS} D. Yang and S. Yang, \textit{Weighted local Orlicz-Hardy spaces with applications to pseudo-differential operators}, Dissertationes Math. (Rozprawy Mat.), 478 (2011), 1-78.
\end{thebibliography}
\end{document}